\documentclass[12pt]{article}
\usepackage{amsmath, amsthm, amssymb, mathrsfs}
\usepackage{graphicx}
\usepackage[all,2cell]{xy}
\UseAllTwocells
\usepackage[all]{xy}
\usepackage{hyperref}
\usepackage{graphicx}
\usepackage{manfnt}

\font \boldfrak eufb10
\font \normalfrak eufm10

\def\fr#1{\hbox{\normalfrak #1}}
\def\bfr#1{\hbox{\boldfrak #1}}

\def\bA{{\bf A}}
\def\bB{{\bf B}}

\def\bG{{\bf G}}
\def\bH{{\bf H}}

\def\bM{{\bf M}}
\def\bN{{\bf N}}

\def\bS{{\bf S}}
\def\bT{{\bf T}}
\def\bU{{\bf U}}

\def\bZ{{\bf Z}}

\def\aq{/  \kern-.25em / }

\def\g{\mathfrak{g}}

\def\cH{\mathcal{ H}}

\def\Z{\mathbb{ Z}}

\def\cO{\mathcal{O}}

\def\boxit#1{\vbox{\hrule\hbox{\vrule\kern3pt
          \vbox{\kern3pt#1\kern3pt}\kern3pt\vrule}\hrule}}

\begin{document}

\newtheorem{theorem}{Theorem}[subsection]
\newtheorem{lemma}[theorem]{Lemma}
\newtheorem{proposition}[theorem]{Proposition}
\newtheorem{problem}[theorem]{Problem}
\newtheorem{corollary}[theorem]{Corollary}

\theoremstyle{definition}
\newtheorem{definition}[theorem]{Definition}
\newtheorem{example}[theorem]{Example}
\newtheorem{xca}[theorem]{Exercise}

\theoremstyle{remark}
\newtheorem{remark}[theorem]{\bf Remark}

\def\goth{\frak}

\def\GL{{\rm GL}}
\def\tr{{\rm tr}\, }
\def\A{{\Bbb A}}
\def\bs{\backslash}
\def\Q{{\Bbb Q}}
\def\R{{\Bbb R}}
\def\Z{{\Bbb Z}}
\def\C{{\Bbb C}}
\def\SL{{\rm SL}}
\def\cS{{\cal S}}
\def\cH{{\cal H}}
\def\G{{\Bbb G}}
\def\F{{\Bbb F}}
\def\cF{{\cal F}}

\def\cB{{\cal B}}
\def\cA{{\cal A}}
\def\cE{{\cal E}}

\newcommand{\oB}{{\overline{B}}}
\newcommand{\oN}{{\overline{N}}}

\def\CC{{\Bbb C}}
\def\ZZ{{\Bbb Z}}
\def\QQ{{\Bbb Q}}
\def\cS{{\cal S}}

\def\Ad{{\rm Ad}}

\def\bG{{\bf G}}
\def\bH{{\bf H}}
\def\bT{{\bf T}}
\def\bM{{\bf M}}
\def\bB{{\bf B}}
\def\bN{{\bf N}}
\def\bS{{\bf S}}
\def\bZ{{\bf Z}}
\def\t{\kern.1em {}^t\kern-.1em}
\def\cc#1{C_c^\infty(#1)}
\def\Fx{F^\times}
\def\half{\hbox{${1\over 2}$}}
\def\T{{\Bbb T}}
\def\Ox{{\frak O}^\times}
\def\Ex{E^\times}
\def\Ecl{{\cal E}}

\def\g{{\frak g}}
\def\h{{\frak h}}
\def\k{{\frak k}}
\def\ft{{\frak t}}
\def\n{{\frak n}}
\def\b{{\frak b}}

\def\2by2#1#2#3#4{\hbox{$\bigl( 
{#1\atop #3}{#2\atop #4}\bigr)$}}
\def\wh{\Xi}
\def\C{{\Bbb C}}
\def\bs{\backslash}
\def\adots{\mathinner{\mkern2mu
\raise1pt\hbox{.}\mkern2mu
\raise4pt\hbox{.}\mkern2mu
\raise7pt\hbox{.}\mkern1mu}}

\def\tim{\bf}
\def\timit{\it}
\def\timsm{\scriptstyle}
\def\secttt#1#2{\vskip.2in\noindent
{$\underline{\hbox{#1}}$\break (#2)}\vskip.2in}
\def\sectt#1{\vskip.2in\noindent
{$\underline{\hbox{#1}}$}\vskip.2in}
\def\sectm#1{\vskip.2in\noindent
{$\underline{#1}$}\vskip.2in}
\def\mat#1{\left[\matrix{#1}\right]}
\def\cc#1{C_c^\infty(#1)}
\def\ds{\displaystyle}
\def\Hom{\mathop{Hom}\nolimits}
\def\Ind{\mathop{Ind}\nolimits}
\def\bs{\backslash}
\def\ni{\noindent}
\def\eb{{\bf e}}
\def\fb{{\bf f}}
\def\hb{{\bf h}}
\def\rg{{\goth R}}
\def\ig{{\goth I}}
\def\ccl{{\cal C}}
\def\dcl{{\cal D}}
\def\ecl{{\cal E}}
\def\hcl{{\cal H}}
\def\ocl{{\cal O}}
\def\ncl{{\cal N}}
\def\ag{{\goth a}}
\def\bg{{\goth b}}
\def\cg{{\goth c}}
\def\og{{\goth O}}
\def\hg{{\goth h}}
\def\lg{{\goth l}}
\def\mg{{\goth m}}
\def\Og{{\goth O}}
\def\rg{{\goth r}}
\def\sg{{\goth s}}
\def\Sg{{\goth S}}
\def\tg{{\goth t}}
\def\zg{{\goth z}}
\def\C{{\Bbb C}}
\def\Q{{\Bbb Q}}
\def\R{{\Bbb R}}
\def\T{{\Bbb T}}
\def\Z{{\Bbb Z}}
\def\t{{}^t\kern-.1em}
\def\tr{\hbox{tr}}
\def\ad{{\rm ad}}
\def\Ad{{\rm Ad}}
\def\2by2#1#2#3#4{\hbox{$\bigl( {#1\atop #3}{#2\atop #4}\bigr)$}}

\def\Kcl{{\cal K}}
\def\Pcl{{\cal P}}
\def\Scl{{\cal S}}
\def\Vcl{{\cal V}}
\def\Wcl{{\cal W}}
\def\Ecl{{\cal E}}
\def\A{{\Bbb A}}
\def\F{{\Bbb F}}
\def\T{{\Bbb T}}
\def\go{{\goth O}}
\def\Ox{{\goth O}^\times}
\def\Fx{F^\times}
\def\Ex{E^\times}
\def\half{\hbox{${1\over 2}$}}
\def\vtwo{\vskip .2in}
\def\BibFH{{\bf [1]}}
\def\BibGod{{\bf [2]}}
\def\BibCrelle{{\bf [3]}}
\def\BibAmJ{{\bf [4]}}
\def\BibDuke{{\bf [5]}}
\def\BibJL{{\bf [6]}}
\def\BibRR{{\bf [7]}}

\def\BibFH{{\bf [1]}}
\def\BibGod{{\bf [2]}}
\def\BibCrelle{{\bf [3]}}
\def\BibAmJ{{\bf [4]}}
\def\BibDuke{{\bf [5]}}
\def\BibJL{{\bf [6]}}
\def\BibRR{{\bf [7]}}

\newcommand{\N}{\mathbb{N}}

\newcommand{\gen}{{\operatorname{gen}}}
\newcommand{\ind}{\operatorname{ind}}
\newcommand{\Wh}{\mathcal{W}}
\newcommand{\Kr}{\mathcal{K}}
\newcommand{\V}{\mathcal{V}}
\newcommand{\U}{\mathcal{U}}
\renewcommand{\O}{\mathcal{O}}
\newcommand{\lift}{\goth{g}}
\newcommand{\inv}{\iota}
\newcommand{\supp}{{\operatorname{Supp}}}
\def\cW{{\cal W}}

\title{Constructing Tame Supercuspidal Representations}
\author{Jeffrey Hakim}
\date{\today}
\maketitle
\abstract{A new approach to Jiu-Kang Yu's construction of tame supercuspidal representations of $p$-adic reductive groups is presented.  Connections with the theory of cuspidal  representations of finite groups of Lie type and the theory of distinguished representations are also discussed.}
\tableofcontents

\parskip=.13in

\section{Introduction}

This paper provides a new approach to the construction of  tame supercuspidal representations  of $p$-adic reductive groups.  It grew out of a desire to unify the theory of distinguished  tame supercuspidal representations from \cite{MR2431732}   with  Lusztig's analogous theory for  cuspidal representations   of finite groups of Lie type \cite{MR1106911}.


Our construction may be viewed as a revision of Yu's construction \cite{MR1824988} that associates supercuspidal representations to certain representations of compact-mod-center subgroups.  In the applications to distinguished representations, it plays a role analogous to  Deligne-Lusztig ``induction'' of  cuspidal representations from certain characters of tori \cite{MR0393266} (but our approach  is not cohomological). 

Yu's tame supercuspidal representations are parametrized by  complicated objects called generic, cuspidal $G$-data.  (See \cite[Definition 3.11]{MR2431732}.)
The fact that Yu's representations should be associated to simpler objects (such as characters of tori) is already evident  in Murnaghan's paper \cite{MR2767524},  Kaletha's paper \cite{KalYu}, and the general theory of the local Langlands correspondence.

In \cite{MR2767524} and  \cite{KalYu}, generic, cuspidal $G$-data are manufactured from simpler data and then  Yu's construction is applied.
By contrast, our intent is to simplify Yu's construction and make it more amenable to applications, such as the applications to distinguished representations discussed in \S\ref{sec:finitestuff}.

The main source of simplification is the elimination of  Howe factorizations in  Yu's construction.  (See \cite{MR0492087} and \cite[\S4.3]{MR2431732}.)  Recall that to construct a tame supercuspidal representation with Yu's construction, one has  to make a noncanonical choice of a factorization.  Removing the need for this choice, as we do, allows one to more easily develop applications, because one no longer needs technical lemmas proving the existence of  application-friendly factorizations.

Another technical notion required in Yu's construction is that of ``special isomorphisms.''  Though it is already known that special isomorphisms can be chosen canonically and their choice does not affect the isomorphism class of the constructed representation, special isomorphisms have remained a misunderstood and unpleasant aspect of the theory of tame supercuspidal representations.  (Here, we are really referring to the ``relevant'' special isomorphisms of \cite[Definition 3.17]{MR2431732}.  See \S3.4, especially Proposition 3.26, \cite{MR2431732} for more details.  Yu's original definition of ``special isomorphism'' is in \cite[\S10]{MR1824988}.)

Our construction takes as its input a suitable representation $\rho$ (called a ``permissible representation'') of a compact-mod-center subgroup of our $p$-adic group $G$ and then  canonically constructs the character of a representation $\kappa$ of an open compact-mod-center subgroup such that  $\kappa$ induces a supercuspidal representation $\pi$ of $G$. 
The construction of the character of $\kappa$ (and hence the equivalence classes of $\kappa$ and $\pi$) is canonical, reasonably simple, and makes no reference to special isomorphisms.  We do, in fact, use special isomorphisms to establish the existence of a representation $\kappa$ with the given character, but the meaning and role of special isomorphisms in the theory becomes more transparent.

We also observe that, in many cases, there is no need to explicitly impose genericity conditions in  our construction, but, instead, the necessary genericity properties are automatic. 
We are indebted to Tasho Kaletha for explaining to us that Yu's main genericity condition, condition {\bf GE1}, is built into our construction.

In general, the influence of Kaletha's ideas on this paper is hard to adequately acknowledge.
After the initial draft of the paper was circulated, the author became aware of Kaletha's work \cite{KalYu} on regular supercuspidal representations.  This prompted a major revision of the paper, during which various tameness conditions and other technicalities were removed.

The best evidence of the utility of our construction lies in the proof of Theorem \ref{motivation}, a result that is  stated in this paper  but proven   in a companion paper \cite{ANewHM}.  The proof uses the theory in this paper to establish stronger versions of the author's main results with Fiona Murnaghan, with considerably less effort.

Theorem \ref{motivation} links  the representation theories of finite and $p$-adic groups by providing a uniform formula for the dimensions of the spaces of invariant linear forms for distinguished representations.

For finite groups of Lie type, our formula is essentially a reformulation of the main result of \cite{MR1106911}.  For $p$-adic reductive groups, it is a reformulation of  the main result of  \cite{MR2431732}.
In both cases, these reformulations refine earlier reformulations in \cite{MR2925798}.


This paper is structured in an unorthodox way to accommodate a variety of readers, each of whom is only interested in select aspects of the theory. 
The construction itself is expeditiously explained (without proofs)
in \S\ref{sec:construction}.  The main result is Theorem \ref{ourmainresult}.
The proofs are  taken up in \S\ref{sec:techdetails}.
For the most part, we expect that the connection between our construction and Yu's construction will be self-evident in \S\ref{sec:construction}, but  \S\ref{sec:YuisoftenMe} formalizes this connection.  Finally, in \S\ref{sec:finitestuff}, we describe the connection between our construction and the Deligne-Lusztig construction in applications to the theory of distinguished representations.

We gratefully acknowledge the advice and generous help of Jeffrey Adler and Joshua Lansky throughout the course of this project. 

\section{The construction}\label{sec:construction}

\subsection{The inducing data}

Let $F$ be a finite extension of a field $\mathbb{Q}_p$ of $p$-adic numbers for some odd prime $p$, 
 and let $\overline{F}$ be an algebraic closure of $F$.
 
 Let $\bG$ be a connected reductive $F$-group.
 We are interested in constructing supercuspidal representations $\pi$ of $G = \bG(F)$.
 
The input for our construction is a suitable representation $(\rho , V_\rho)$ of a suitable compact-mod-center subgroup $H_x$ of $G$, and the map from $\rho$ to $\pi$ has some formal similarities to Deligne-Lusztig induction of cuspidal representations of finite groups of Lie type.

Let us now specify $H_x$ and $\rho$ more precisely.
 
 Fix an $F$-subgroup $\bH$ of $\bG$ that is a Levi subgroup of $\bG$ over some tamely ramified finite extension of $F$.
 We  require that the quotient $\bZ_\bH/\bZ_\bG$ of the centers of $\bH$ and $\bG$ is $F$-anisotropic.  

Fix  a vertex $x$ in the reduced building $\mathscr{B}_{\rm red}(\bH,F)$.  Let $H = \bH (F)$ and let $H_x$ be the stabilizer of $x$ in $H$.  Let $H_{x,0}$ be the corresponding (maximal) parahoric subgroup in $H$ and let $H_{x,0+}$ be its prounipotent radical.
 
Note that the reduced building $\mathscr{B}_{\rm red}(\bH,F)$ of $\bH$ is, by definition, identical to the extended building $\mathscr{B}(\bH_{\rm der} , F)$ of the derived group $\bH_{\rm der}$ of $\bH$.

Let $\bH_{\rm sc}\to \bH_{\rm der}$ be the universal cover of $\bH_{\rm der}$.
We identify the  buildings $\mathscr{B} (\bH_{\rm sc},F)$ and $\mathscr{B}(\bH_{\rm der},F)$   (as in \S\ref{sec:zext}) and let $H_{{\rm der},x,0+}^\flat$ denote the image of $H_{{\rm sc},x,0+}$ in $\bH_{\rm der}$.
 
 \bigskip
\begin{definition}\label{permissibledef} A representation of $H_x$ is {\bf permissible} if it is a smooth, irreducible, complex representation $(\rho ,V_\rho)$ of $H_x$ such that:
\begin{itemize}
\item[(1)] $\rho$ induces an irreducible (and hence supercuspidal) representation of $H$,
\item[(2)] the restriction of $\rho$ to $H_{x,0+}$ is a multiple of some character $\phi$ of $H_{x,0+}$, 
\item[(3)] $\phi$ is trivial on $H_{{\rm der},x,0+}^\flat$,
\item[(4)]  the  dual cosets $(\phi |Z^{i,i+1}_{r_i})^*$ (defined in \S\ref{sec:dualcoset}) contain elements that satisfy Yu's condition {\bf GE2}  (stated below in \S\ref{sec:GE2}).
\end{itemize}
\end{definition}

\bigskip

\begin{remark}\label{diszerorem}
When the order of the fundamental group $\pi_1(\bH_{\rm der})$ of $\bH_{\rm der}$ is not divisible by $p$, then we show in Lemma \ref{goodp} that $H_{{\rm der},x,0+}^\flat=H_{{\rm der},x,0+}$.  In this case, we show in Lemma \ref{goodptwo} that every permissible representation
$\rho$ may be expressed as $\rho_0 \otimes (\chi|H_x)$, where $\rho_0$ has depth zero and $\chi$ is a quasicharacter of $H$ that extends $\phi$.  (So our representations $\rho_0$ are essentially the same as the representations denoted by $\rho$ in \cite{MR1824988}.)  This allows one to appeal to the depth zero theory of Moy and Prasad \cite{MR1253198, MR1371680} to explicitly describe the representations $\rho$.  Even when $p$ divides the order of $\pi_1 (\bH_{\rm der})$, one can apply depth zero Moy-Prasad theory by lifting to a $z$-extension, as described in \S\ref{sec:zext}.  In this case, one obtains a factorization  over the $z$-extension, but the factors may not correspond to representations of $H_x$.
\end{remark}

\medskip
\begin{remark}\label{conditionfourrem}
In light of \cite[Lemma 8.1]{MR1824988}, condition (4) is only required when 
 $p$ is a torsion prime for the dual of the root datum of $\bG$ (in the sense of  \cite[\S7]{MR1824988} and \cite[\S2]{MR0354892}). 
 Using Steinberg's results (Lemma 2.5 and Corollary 1.13 of \cite{MR0354892}), we can make this more explicit as follows.  Since we assume $p\ne 2$, condition (4) is  needed precisely when 
\begin{itemize}
\item
$p=3$ and $\bG_{\rm der}$ has a factor of type $E_6$, $E_7$, $E_8$ or $F_4$,
\item
$p=5$ and $\bG_{\rm der}$ has a factor of type $E_8$, or 
\item  $\bG_{\rm der}$ has a factor of type $A$ and
$p$ is a torsion prime for $X/\Z \Phi$, where $\Z\Phi$ is the root lattice and $X$ is the character lattice relative to some maximal torus of $\bG$.
\end{itemize}
  In particular, condition (4) is unnecessary for  $\GL_n$, but is required for $\SL_n$.
\end{remark}

\medskip
\begin{remark}
A given permissible representation can correspond to supercuspidal representations for two different groups.  For example, suppose $\Psi = (\vec\bG,y,\rho_0,\vec\phi)$ is a generic cuspidal $G$-datum with $d>0$ and $\phi_d =1$, with notations as in \cite{MR2431732}.  Let $\Psi'$ be the datum obtained from $\Psi$ by deleting the last components of $\vec\bG$ and $\vec\phi$.  Then the tame supercuspidal representations of $G$ and $G^{d-1}$ associated to $\Psi$ and $\Psi'$, respectively, are both associated to the same permissible representation $\rho_0\otimes (\prod_{i=0}^{d-1}(\phi_i|G^0_{[y]}))$.
\end{remark}

\subsection{Basic notations}

We use boldface letters for $F$-groups and non-boldface for the groups of $F$-rational points.  Gothic letters are used for  Lie algebras, and  the subscript ``der'' is used for derived groups of algebraic groups and their Lie algebras.  For example, $G = \bG (F)$,   $\bfr{g} = {\rm Lie}(\bG)$, $\mathfrak{g} = \bfr{g}(F)$, and $\bG_{\rm der}$ is the derived group of $\bG$.  We let $G_{\rm der}$ denote $\bG_{\rm der}(F)$, but caution that this is not necessarily the derived group of $G$.

In general, our notations tend to follow \cite{MR2431732}, which, in turn, mostly follows \cite{MR1824988}.   This applies, in particular, to groups $G_{y,f}$ associated to concave functions, such as  Moy-Prasad groups $G_{y,s}$ and groups $\vec G_{y,\vec s}$ associated to twisted Levi sequences $\vec\bG$ and admissible sequences $\vec s$.  Many of these notations are recalled in  \S\ref{sec:expsection}, as well as conventions regarding variants of exponential maps (such as Moy-Prasad isomorphisms).
We use colons in our notations for quotients, for example,
 $G_{y,s:r} = G_{y,s}/G_{y,r}$.  Similar notations apply to the Lie algebra filtrations and to $F$-groups other than $\bG$. 

Moy-Prasad filtrations  over $F$ are defined with respect to the standard valuation $v_F$ on $F$.  Moy-Prasad filtrations over extensions of $F$ are also defined with respect to $v_F$.   Thus we have intersection formulas such as $\bG (E)_{y,r}\cap G = G_{y,r}$.

Let  $$E_r =E\cap v_F^{-1}([r,\infty])$$ and, for $r>0$,
$$E_r^\times = 1+ E_r.$$

Given a point $y$ in the reduced building $\mathscr{B}_{\rm red}(\bG , F)$, we let $\mathsf{G}_y$ denote the associated reductive group defined over the residue field $\fr{f}$ of $F$.  Thus $\mathsf{G}_y(\fr{F}) = G(F^{\rm un})_{y,0:0+}$, where $F^{\rm un}$ is the maximal unramified extension of $F$ in $\overline{F}$, and $\fr{F}$ is the residue field of $F^{\rm un}$ (which is an algebraic closure of $\fr{f}$).

\subsection{The torus $\bT$ and the subtori $\bT_{\mathscr{O}}$}
The construction requires the choice of a maximal, elliptic $F$-torus $\bT$ in $\bH$ with the   properties:
\begin{itemize}
\item The splitting field $E$ in $\overline{F}$ of $\bT$ over $F$ is a tamely ramified extension of $F$.  (Note that $E/F$ is automatically a finite Galois extension.)
\item The point $x$ is the unique fixed point of ${\rm Gal}(E/F)$ in the apartment $\mathscr{A}_{\rm red}(\bH,\bT, E)$.
\end{itemize}

\medskip
\begin{remark}In choosing $\bT$, we are appealing to some known facts:
\begin{itemize}
 \item $\mathscr{B}_{\rm red}(\bH , F)= \mathscr{B}_{\rm red}(\bH,E)^{{\rm Gal}(E/F)}$, since $E/F$ is tamely ramified,  according to a result of Rousseau \cite{MR0491992}. (See also \cite{MR1871292}.)
 \item $\mathscr{A}_{\rm red} (\bH,\bT,F)$ consists of a single point, since  $\bT$ is elliptic. 
\item Not only do tori $\bT$ of the above kind exist, but DeBacker \cite{MR2214792} gives a construction of a special class of such tori:\begin{itemize}
\item Choose an elliptic maximal $\fr{f}$-torus $\mathsf{T}$ in $\mathsf{H}_x$.\hfill\break  (See Lemma 2.4.1 \cite{MR2214792}.)  \item  Choose a maximal $F^{\rm un}$-split torus $\bS$ in $\bH$ such that
\begin{itemize}
\item $\bS$ is defined over $F$, 
\item $x\in \mathscr{A}_{\rm red}(\bH,\bS,F^{\rm un})$, and 
\item $\mathsf{T}(\fr{F})$ is the image of $\bS(F^{\rm un})\cap \bH (F^{\rm un})_{x,0}$ in $\mathsf{H}_x(\fr{F})$.
\end{itemize}
(See Lemma 2.3.1 \cite{MR2214792}.)
 \item
Take $\bT$ to be the centralizer  of $\bS$ in $\bH$.  
\end{itemize}
Such tori $\bT$ are called ``maximally unramified elliptic maximal tori'' in \S3.4.1 \cite{KalYu}, and  various reformulations of the definition are provided there.  For example, an elliptic maximal $F$-torus $\bT$ is maximally unramified  if it is  contained in a Borel subgroup of $\bH$ over $F^{\rm un}$.  Lemma 3.4.4 \cite{KalYu} says that any two maximally unramified elliptic maximal tori in $\bH$ corresponding to the same point in $\mathscr{B}_{\rm red}(\bH,F)$ must be $H_{x,0+}$-conjugate.
Lemma 3.4.18 \cite{KalYu} says every regular depth zero supercuspidal representation of $H$ comes from a regular depth zero character of the $F$-rational points of a maximally unramified elliptic maximal torus of $\bH$.
\end{itemize}
\end{remark}

\medskip
\begin{remark} It is necessarily the case that $T\subset H_x$ since $T$ preserves both $\mathscr{A}_{\rm red}(\bH,\bT,E)$ and $\mathscr{B}_{\rm red} (\bH,F)$, and the intersection of the latter spaces is $\{ x\}$.
\end{remark}

\medskip
\begin{remark} When $H_{\rm der}$ is compact, every $F$-torus in $\bH$ is elliptic and $\mathscr{B}_{\rm red} (\bH,F)$ is a point $x$ and so $H_x =H$.  In this case, every maximal $F$-torus $\bT$ in $\bH$ is elliptic and $x$ is (trivially) the unique ${\rm Gal}(E/F)$-fixed point in $\mathscr{A}_{\rm red}(\bH,\bT,E)$.  This example illustrates, in the extreme, that the point $x$ does not  determine the torus $\bT$.
\end{remark}


%

Let $\Phi = \Phi (\bG,\bT)$ and $\Gamma = {\rm Gal}(\overline{F}/F)$.  Given a $\Gamma$-orbit $\mathscr{O}$ in $\Phi$, let $\bT_{\mathscr{O}}$ be the torus generated by the tori $\bT_a = \text{image}(\check a)$ as $a$ varies over $\mathscr{O}$.

For positive depths, the Moy-Prasad filtration of $T_{\mathscr{O}} = \bT_{\mathscr{O}}(F)$ is given using the norm map
 \begin{eqnarray*}
N_{E/F}: \bT_{\mathscr{O}}(E)&\to& T_{\mathscr{O}} \\
t&\mapsto& \prod_{\gamma\in {\rm Gal}(E/F)} \gamma (t)
\end{eqnarray*}
as follows.  If $a\in \mathscr{O}$ and $r>0$ then, according to Lemma \ref{painintheneck},  $$T_{\mathscr{O},r} =N_{E/F}(\check a (E_r^\times))= N_{E/F}(\bT_a(E)_r)= N_{E/F}(\bT_{\mathscr{O}}(E)_r),$$ for all $a\in\mathscr{O}$.

\bigskip\begin{remark}
Though the tori $\bT_{\mathscr{O}}$ are not explicitly considered in \cite{KalYu}, the norm map on $\bT_a(E)$ plays a prominent role there, and our use of the norm map has  been influenced by the theory in \cite{KalYu}.  (See \cite[\S3.7]{KalYu}.)
\end{remark}

\subsection{Recovering $\vec r$ and $\vec\bG$}\label{sec:recovery}

A key  step in our construction is to recursively construct from $\phi$  a sequence $$\vec\bG = (\bG^0,\dots , \bG^d)$$of subgroups
$$\bH = \bG^0\subsetneq \cdots \subsetneq \bG^d=\bG$$
and a sequence
$$\vec r = (r_0,\dots ,r_d)$$
of real numbers
$$0\le r_0 <\cdots < r_{d-1} \le r_{d}.$$

We construct the groups recursively in the order $\bG^d,\dots ,\bG^0$ (and similarly for $r_d,\dots , r_0$), so $d$ should be treated initially as an unknown nonnegative integer whose value only becomes evident at the end of the recursion.
This indexing is compatible with \cite{MR1824988}, but it is the opposite of what is done in \cite{MR0492087}.

To begin the recursion, we take  $\bG^d = \bG$ and 
$$r_d = \text{depth}(\rho) = \text{depth}(\phi).$$

When $\bG = \bH$, we declare that $d=0$.

Now assume $\bH\subsetneq \bG$.  Suppose $i\in \{ 0,\dots , d-1\}$ and $\bG^{i+1}$ has been defined and strictly contains $\bH$.

In general, given $\bG^{i+1}$, we take
$$\bH^i = \bG^{i+1}_{\rm der}\cap \bH,\qquad \bT^i = \bG^{i+1}_{\rm der}\cap \bT.$$
Note that $\bT^i$ is the
torus generated by the tori $\bT_{\mathscr{O}}$ associated to \break $\Gamma$-orbits $\mathscr{O}$ in $\Phi (\bG^{i+1},\bT)$.  
(See \cite[Proposition 8.1.8(iii)]{MR1642713}.)

To say $\bT$ is elliptic in $\bH$ means that $\bT/\bZ_\bH$ is $F$-anisotropic or, equivalently, $\bH_{\rm der}\cap \bT$ is $F$-anisotropic.  Since we assume $\bZ_\bH/\bZ_\bG$ is $F$-anisotropic, $\bT$ must be elliptic in $\bG$.  So $\bT^{d-1}$, and hence $\bT^i$, must be $F$-anisotropic.
It follows that $x$ determines unique points (that we also denote $x$) in $\mathscr{B}_{\rm red}(\bG^{i+1},F)$ and $\mathscr{B}_{\rm red} (\bH^i,F)$.  

\bigskip
\begin{remark}
Our conventions regarding embeddings of buildings are similar to those in \cite{MR2431732}.  See, in particular, Remark 3.3 \cite{MR2431732}, regarding the embedding of $\mathscr{B}_{\rm red}(\bG^i,F)$ in $\mathscr{B}_{\rm red}(\bG^{i+1},F)$.
It should be understood that the images of $x$ in the various buildings  depends on the choice of $\bT$.  As in \cite{MR1824988}, it will be easy to see that varying $\bT$ does not change the isomorphism class of the resulting supercuspidal representation.  (See the discussion after Lemma 1.3 \cite{MR1824988}.)
\end{remark}

At this point, we encounter a technical issue involving the residual characteristic $p$.

Suppose first that $p$ does not divide the order of the fundamental group $\pi_1(\bG_{\rm der})$.  
Then we define
$$r_i = \text{depth}(\phi|H^i_{x,0+}) = \text{depth}(\phi |T^i_{0+}).$$
The latter identity of depths is a consequence of Lemma \ref{HTtransfer}.

If $p$ divides the order $\pi_1(\bG_{\rm der})$, we can simply replace $\bG$ by a $z$-extension $\bG^\sharp$ of the type considered in \S\ref{sec:zext}.  In other words, we  identify $G$ with $G^\sharp /N$, where $\bN$ is the kernel of the $z$-extension, and  apply our construction to $\bG^\sharp$ instead of $\bG$.

Though it is somewhat tedious, the latter case can also be translated into terms that are intrinsic to $G$.  (In other words, one does not need to refer to the $z$-extension, but one does need to refer to the universal cover of $\bG_{\rm der}$.)
To appreciate the source of complications, consider, for example, the character $\phi$ of $H_{x,0+}$.  According to Lemma 3.5.3 \cite{KalYu}, we have a surjection $H^\sharp_{x,0+}\to H_{x,0+}$, and thus the pullback $\phi^\sharp$ of $\phi$ to $H^\sharp_{x,0+}$ captures all of the information carried by $\phi$.  
But we may not have similar surjections associated to the groups $H^i_{x,0+}$ and $T^i_{0+}$, and this creates difficulties.  The appropriate depths $r_i$ for the construction may no longer be the depths of the restrictions of $\phi$ to $H^i_{x,0+}$ and $T^i_{0+}$.  Therefore, we cannot simply define the $r_i$'s as above.

In light of the previous remarks, we generally assume $p$ does not divide the order of $\pi_1(\bG_{\rm der})$ when we describe our construction, though we do consider more general $p$ when considering the complete family of representations that the construction captures.  (See Lemma \ref{IamYu}.)

Continuing with the construction, we now take $\bG^i$ to be the (unique) maximal subgroup of $\bG$ such that:
\begin{itemize}
\item $\bG^i$ is defined over $F$.
\item $\bG^i$ contains $\bH$.
\item $\bG^i$ is a Levi subgroup of $\bG$ over $E$.
\item $\phi |H^{i-1}_{x,r_i}=1$.
\end{itemize}

(It follows from Lemma \ref{HTtransfer} that the condition $\phi |H^{i-1}_{x,r_i}=1$ is equivalent to the condition $\phi |T^{i-1}_{r_i}=1$.)

As soon as $\bG^i = \bH$, the recursion ends.  The value of $i$ for which $\bG^i=\bH$ is declared to be $i=0$.  The values of $d$ and the other indices also become evident at this point.

\medskip
\begin{remark}  There is an alternate way to construct $\vec r$ and $\vec\bG$ that is direct (as opposed to being recursive) and clarifies the meaning of ${\bf G}^i$ and $r_i$, as well as the fact that these objects are well defined.  It exploits the choice of $\bT$ and is inspired by \S3.7 \cite{KalYu}.
One can define $$r_0<\cdots < r_{d-1}$$ to be the sequence of positive numbers that occur as depths of the various characters $\phi |T_{\mathscr{O},0+}$ for $\mathscr{O}\in \Gamma\bs \Phi$.  (This does not preclude the possibility that $\phi |T_{\mathscr{O},0+}=1$ for some orbits $\mathscr{O}$.)
Next, when $i<d$, we take $$\Phi^i = \bigcup_{\mathscr{O}\in \Gamma\bs \Phi,\ \phi|T_{\mathscr{O},r_i}=1} \mathscr{O}.$$  One can define $\bG^i$ to be the unique $E$-Levi subgroup of $\bG$ that contains $\bH$ and is such that $\Phi (\bG^i,\bT) = \Phi^i$.
\end{remark}

\medskip
\begin{remark} To connect our narrative with Yu's construction, it might be helpful to think in terms of Howe's theory of factorizations of admissible quasicharacters (and its descendants), even though our approach avoids factorization theory.  In the present context,  a factorization of $\phi$ would consist of quasicharacters $\phi_i : G^i \to \C^\times$ such that $\phi = \prod_{i=0}^d (\phi_i|H_{x,0+})$.  But the factors $\phi_{i+1},\dots , \phi_d$ are trivial on $[G^{i+1},G^{i+1}]\cap H_{x,0+}$ and the factors $\phi_0,\dots , \phi_{i-1}$ are trivial on $H_{x,r_i}$.  Therefore, on $[G^{i+1},G^{i+1}] \cap H_{x,r_i}$, the character $\phi$ coincides with the quasicharacter $\phi_i$.  Lemma \ref{goodp} (with $\bH$ replaced by $\bG^{i+1}$) implies  that $[G^{i+1},G^{i+1}] \cap H_{x,r_i} = H^i_{x,r_i}$, so long as $p$ does not divide the order of the fundamental group of $\bG^{i+1}_{\rm der}$.  
Hence, $\phi|H^i_{x,r_i} = \phi_i|H^i_{x,r_i}$.  Our present approach to constructing supercuspidal representations is partly motivated by the heuristic that the essential information carried by the factor $\phi_i$ is contained in the restriction $\phi_i|H^i_{x,r_i}$ and, since this restriction agrees with $\phi |H^i_{x,r_i}$, there  is no need to use factorizations.  (See Lemma \ref{IamYu} for more details.)
\end{remark}

\medskip
\begin{remark}
Sequences such as $\vec\bG = (\bG^0,\dots ,\bG^d)$ are called\break ``tamely ramified twisted Levi sequences'' (Definition 2.43 \cite{MR2431732}), and one may view $\bG^{i+1}$ as being constructed from $\bG^i$ by adding unipotent elements.  By contrast, if $\bH^d = \bH$ and $\vec\bH = (\bH^0,\dots ,\bH^d)$ then each $\bH^{i+1}$ is obtained from $\bH^i$ by adding semisimple elements.
\end{remark}

\subsection{The inducing subgroup $K$}

The objective of our construction is to associate to $\rho$ an equivalence class of supercuspidal representations of $G$.  These representations are induced from a certain compact-mod-center subgroup $K$ of $G$ that is defined in this section.  Later, we define an equivalence class of representations $\kappa$ of $K$ by canonically defining the (common) character of these representations $\kappa$.  The representation $\pi$ of $G$ induced from such a $\kappa$ will lie in the desired equivalence class of supercuspidal representations $\pi$ associated to $\rho$.

For each $i\in \{ 0,\dots , d-1\}$, let $$s_i= r_i/2$$
and
\begin{eqnarray*}
J^{i+1} &=& (G^i,G^{i+1})_{x,(r_i,s_i)}\\
J^{i+1}_+ &=& (G^i,G^{i+1})_{x,(r_i,s_i+)}\\
J^{i+1}_{++} &=& (G^i,G^{i+1})_{x,(r_i+,s_i+)}.
\end{eqnarray*}
(The notations on the right hand side follow \cite[page 53]{MR2431732}.
See also the discussion in \S\ref{sec:expsection} below regarding the groups $\vec\bG_{x,\vec t}$ associated to the twisted Levi sequence $\vec\bG = (\bG^i,\bG^{i+1})$  and admissible sequences of numbers.)

Our desired inducing subgroup $K$ is 
given by 
$$K= H_xJ,$$
where
$$J = J^1\cdots J^d.$$

\subsection{The subgroup $K_+$}

The (equivalent) representations $\kappa$ of $K$ mentioned above will turn out to restrict to a multiple of a certain canonical character $\hat\phi$ of a certain subgroup $K_+$.
 The subgroup $K_+$ is 
given by 
$$K_+ = H_{x,0+}L_+,$$
where
\begin{eqnarray*}
L_+&=& (\ker \phi)L^1_+\cdots L^d_+,\\
 L^{i+1}_+&=& G^{i+1}_{\rm der}\cap J^{i+1}_{++}.
 \end{eqnarray*}

 The character $\hat\phi$ is the inflation $$\hat\phi = {\rm inf}_{H_{x,0+}}^{K_+}\kern-.3em (\phi )$$ of $\phi$  to $K_+$, that is, $\hat\phi (h\ell) = \phi (h)$, when $h\in H_{x,0+}$ and $\ell\in L_+$.   More details regarding the definition of $\hat\phi$ are provided in \S\ref{sec:hatphiwelldef}.

 \medskip
 \begin{remark}
 The inflation of characters from $H_{x,0+}$ to $K_+$ just considered should not be confused with a different inflation procedure considered in \cite{MR2431732}, where we are using the decomposition $K_+ = H_{x,0+}J^1_+\cdots J^d_+$.
 \end{remark}

\subsection{The dual cosets}\label{sec:dualcoset}

Fix $i\in \{ 0,\dots , d-1\}$ and define $\bZ^{i,i+1}$ to be the 
$F$-torus
$$\bZ^{i,i+1} = (\bG^{i+1}_{\rm der}\cap \bZ^i)^\circ,$$ where $\bZ^i$ is the center of $\bG^i$.

Fix a character $\psi$ of $F$ that is trivial on the maximal ideal $\mathfrak{P}_F$ but nontrivial on the ring of integers $\mathfrak{O}_F$.

\medskip
\begin{definition}\label{dualcosetdef} The {\bf dual coset} $ (\phi |Z^{i,i+1}_{r_i})^*$ of $\phi |Z^{i,i+1}_{r_i}$ is the coset in $\fr{z}^{i,i+1,*}_{-r_i:(-r_i)+}$ consisting of elements
 $X^*_i$ in $\fr{z}^{i,i+1,*}_{-r_i}$ that represent $\phi |Z^{i,i+1}_{r_i}$ in the sense that
$$\phi ( \exp (Y+\fr{z}^{i,i+1}_{r_i+})) = \psi (X^*_i(Y)),\quad \forall Y\in \fr{z}^{i,i+1}_{r_i}.$$
\end{definition}

(The notation ``$\exp$'' here is for the Moy-Prasad isomorphism on $\fr{z}^{i,i+1}_{r_i:r_i+}$.  See \S\ref{sec:expsection} for more details.)

Let $\fr{z}^{i,i+1}$ be the Lie algebra of $Z^{i,i+1}$, and let $\fr{z}^{i,i+1,*}$ be the dual of $\fr{z}^{i,i+1}$.
The decomposition $$\fr{g}^i = \fr{z}^i\oplus \fr{g}^i_{\rm der}$$ restricts to a decomposition $$\fr{h}^i = \fr{z}^{i,i+1} \oplus \fr{h}^{i-1}.$$

Using this, we associate to each coset in $\fr{z}^{i,i+1,*}_{-r_i:(-r_i)+}$ 
a character of $H^i_{x,r_i:r_i+}$ that is trivial on $H^{i-1}_{x,r_i:r_i+}$.  The character of $H^i_{x,r_i:r_i+}$ associated to the dual coset $(\phi |Z^{i,i+1})^*$ is just the restriction of $\phi$ since $\phi |H^{i-1}_{x,r_i}=1$.

In other words, if $X\in \fr{z}^{i,i+1}_{r_i}$ and $Y\in \fr{h}^{i-1}_{x,r_i}$ then
$$\phi (\exp (X+Y+ \fr{h}^i_{x,r_i+})) = \psi (X^*_i(X)),$$ where $X^*_i$ is any element of $ (\phi |Z^{i,i+1}_{r_i})^*$  and $$\exp: \fr{h}^i_{x,r_i:r_i+}\cong H^i_{x,r_i:r_i+}$$ is the isomorphism defined in \S\ref{sec:expsection}.

Lemma \ref{GEonelemma} establishes that every  element $X^*_i$ of the dual coset  $ (\phi |Z^{i,i+1}_{r_i})^*$ is $\bG^{i+1}$-generic of depth $-r_i$.  (Here, we are using condition (4) in Definition \ref{permissibledef}, as well as Lemma 8.1 \cite{MR1824988}.)

\medskip
\begin{remark} The decomposition $\fr{h}^i = \fr{z}^{i,i+1} \oplus \fr{h}^{i-1}$
restricts to
$$\fr{t}^i = \fr{z}^{i,i+1} \oplus \fr{t}^{i-1},$$
where $\fr{t}^i$ and  $\fr{t}^{i-1}$ are generated by the $\fr{t}_\mathscr{O}$'s associated to $\Phi (\bG^{i+1},\bT)$ and $\Phi (\bG^i,\bT)$, respectively.  We caution that $\fr{z}^{i,i+1}$ is usually not generated by $\fr{t}_{\mathscr{O}}$'s and, in fact, may not contain any $\fr{t}_{\mathscr{O}}$'s.
\end{remark}

\subsection{The construction}

For each $i\in \{ 0,\dots ,d-1\}$, the 
elements $X^*_i$ of the dual coset $ (\phi |Z^{i,i+1}_{r_i})^*$ also represent a character of $J^{i+1}_+$ that we denote by $\zeta_i$.  By construction, $\zeta_i$ is the inflation of a $\bG^{i+1}$-generic character of $G^i_{x,r_i}$ of depth~$r_i$.

Since $G^i_{x,r_i:r_i+}$ is isomorphic to a finite vector space over a finite field of characteristic $p$, every character of $G^i_{x,r_i:r_i+}$, in particular $\zeta_i$, must take values in the group $\mu_p$ of complex $p$-th roots of unity.

%
%
%

The space $V$ of our inducing representation $\kappa$ will be a tensor product
$$V = V_\rho\otimes V_0\otimes \cdots \otimes V_{d-1}$$
of $V_\rho$ with spaces $V_i$, such that each $V_i$ is the space of a Heisenberg representation $\tau_i$ attached to the character $\zeta_i$.

The next step is to define these Heisenberg representations.  

Let $W_i$ be the quotient $J^{i+1}/J^{i+1}_+$ viewed as an $\F_p$-vector space with nondegenerate (multiplicative) symplectic form
$$\langle uJ^{i+1}_+ ,vJ^{i+1}_+\rangle = \zeta_i(uvu^{-1}v^{-1}).$$
(The fact that this is a nondegenerate symplectic form is shown in \cite[Lemma 11.1]{MR1824988}.)

Let
$\mathscr{H}_i$ be the (multiplicative) Heisenberg $p$-group  such that
$$\mathscr{H}_i = W_i\times \mu_p$$
with multiplication given by
$$(w_1,z_1)(w_2,z_2) = (w_1w_2, z_1z_2 \langle w_1,w_2\rangle^{(p+1)/2}).$$
Let $(\tau_i,V_i)$ be a Heisenberg representation of $\mathscr{H}_i$ whose central character is the identity map on $\mu_p$.  (Up to isomorphism, $\tau_i$ is unique.)

Let $\mathscr{S}_i$ be the symplectic group ${\rm Sp}(W_i)$.  We let $\mathscr{S}_i$ act on $\mathscr{H}_i$ via its natural action on the first factor of $W_i\times \mu_p$.  The semidirect product $\mathscr{S}_i\ltimes \mathscr{H}_i$ is the Cartesian product $\mathscr{S}_i\times \mathscr{H}_i$ with multiplication
$$(s_1,h_1)(s_2,h_2)= (s_1s_2, (s_2^{-1}h_1)h_2).$$
Except for the case in which $p=3$ and $\dim_{\F_p}(W_i)=2$, there is a unique extension $\hat\tau_i$ of $\tau_i$ to a representation 
$$\hat\tau_i :\mathscr{S}_i\ltimes \mathscr{H}_i\to \GL (V_i).$$
When $p=3$ and $\dim_{\F_p}(W_i)=2$, we also have a canonical lift $\hat\tau_i$ that is specified in \cite[Definition 2.17]{MR2431732}.
Given $h\in H_x$, let $$\omega_i(h) = \hat\tau_i({\rm Int}(h),1),$$ where ${\rm Int}(h)$ comes from the conjugation action of $h$ on $J^{i+1}$.

We now state our main result:

\begin{theorem}\label{ourmainresult}  Suppose $\rho$ is a permissible representation of $H_x$.  
Up to isomorphism, there is a unique representation $\kappa = \kappa (\rho)$ of $K$ such that:
\begin{itemize}
\item[{\rm (1)}]  The character of $\kappa$ has support in $H_xK_+$.
\item[{\rm (2)}]  $\kappa |K_+$ is a multiple of $\hat\phi$.
\item[{\rm (3)}]  $\kappa | H_x = \rho \otimes \omega_0\otimes\cdots \otimes \omega_{d-1}$.
\end{itemize}
Then $\pi (\rho)= {\rm ind}_{K}^G(\kappa (\rho))$ is an irreducible, supercuspidal representation  of $G$ whose isomorphism class is canonically associated to $\rho$.  The isomorphism class of every tame supercuspidal representation of $G$ constructed by Yu in \cite{MR1824988} contains a representation $\pi (\rho)$.
\end{theorem}

\begin{proof}
The fact that the equivalence class of $\kappa (\rho)$ is well-defined and canonically associated to $\rho$ is proven in Lemma \ref{welldefined}.
The proof of the latter result also establishes that $\kappa (\rho)$ is determined by the properties listed above.
The fact that $\pi (\rho)$ is irreducible  is proven in Lemma \ref{irred}.  This implies that $\pi (\rho)$ is also supercuspidal, as explained on page 12 of \cite{MR2431732}.
Lemma \ref{IamYu} establishes the connection between our construction and Yu's construction.\end{proof}

\medskip
\begin{corollary}
If $\rho$ is permissible and $\Theta_\pi$ is the smooth function on the regular set of $G$ that represents the character of  $\pi = \pi(\rho)$ then $\Theta_\pi$ is given  by the Frobenius formula
$$\Theta_\pi (g) =\sum_{C\in K\bs G/K}\ \sum_{h\in C}\dot\chi_\kappa (h gh^{-1}),$$
where $\dot\chi_{\kappa}$ is the function on $G$ defined by
$$\dot\chi_\kappa (jk) = \hat\phi(k)\cdot {\rm tr}(\rho (j))\cdot {\rm tr} (\omega_0(j))\cdots {\rm tr} (\omega_{d-1}(j)),$$ 
for $j\in H_x$ and $k\in K_+$, and $\dot\chi_\kappa \equiv 0$ on $G-H_xK_+$.
In particular, $\Theta_\pi (g) =0$ when $g$ does not lie in the $G$-invariant set generated by $H_xK_+$.
\end{corollary}

We refer to \cite[\S4]{MR1767897} for details regarding the Frobenius formula.

\section{Technical details}\label{sec:techdetails}

This chapter is essentially a collection of technical appendices to the previous chapter.

\subsection{$z$-extensions}\label{sec:zext}

A {\bf $z$-extension} of $\bG$ (over $F$) 
consists of a connected reductive $F$-group $\bG^\sharp$ with simply connected derived group together with an associated exact sequence
$$1\to \bN\to\bG^\sharp\to\bG\to 1$$ of $F$-groups
such that $\bN$ is an induced torus that embeds in the center of $\bG^\sharp$.

To say that $\bN$ is an induced torus means that it is a product of tori of the form ${\rm Res}_{L/F}(\mathbb{G}_m)$, where each $L$ is a finite extension of $F$.  
This implies that the Galois cohomology $H^1(F,\bN)$ is trivial and hence
the cohomology sequence associated to the above exact sequence yields
an exact sequence
$$1\to N\to G^\sharp \to G\to 1$$
of $F$-points.
Therefore every representation of $G$ may be viewed as a representation of $G^\sharp$ that factors through $G$.

Not only do $z$-extensions always exist, but we can always choose a $z$-extension $\bG^\sharp$ of $\bG$ such that $\bN$ is a product of factors ${\rm Res}_{E/F}(\mathbb{G}_m)$, where, as usual, $E$ is the splitting field of $\bT$ over $F$.
In this case, the preimage $\bT^\sharp$ of $\bT$ in $\bG^\sharp$ is an $E$-split maximal torus in $\bG^\sharp$.  
Assume, at this point, that we have fixed such a $z$-extension.

The terminology ``$z$-extension'' and our definition are derived from  \cite[\S1]{MR683003}, but Kottwitz refers to Langlands's article
 \cite[pp.~721-722]{MR540901} as the source of the notion.  Existence results for $z$-extensions can be found in \cite{MR540901} and \cite[pp.~297-299]{MR654325}.

\medskip
{\bf The Philosophy of $z$-extensions.}  {\it When studying the representation theory of general connected reductive $F$-groups, it suffices to consider the representation theory of groups with simply connected derived group.  For groups of the latter type, technicalities involving bad primes (such as those related to the Kneser-Tits problem) are minimized (or eliminated).}

\medskip
Here, we are using the terminology ``bad primes'' to informally refer to any technical issues that might cause a general theory of representations of connected reductive $F$-groups to break down for a finite number of residual characteristics $p$ of the ground field~$F$.

Let $\bH^\sharp$ be the preimage of $\bH$ in $\bG^\sharp$.
Note that we have a natural identification of $\Phi (\bG,\bT)$ with $\Phi (\bG^\sharp,\bT^\sharp)$.  Then the center $Z(\bH)$ is the intersection of the kernels of the roots in $\Phi (\bH,\bT)$.  It is easy to see that the center of $\bH^\sharp$ must be the intersection of the kernels of the roots in $\Phi (\bG^\sharp,\bT^\sharp)$ that correspond to elements of $\Phi (\bH ,\bT)$.
Since $\bH^\sharp$ is the centralizer in $\bG^\sharp$ of the torus $Z(\bH^\sharp)^\circ$, it must be the case that $\bH^\sharp$ is a Levi subgroup of $\bG^\sharp$ (over $E$).  It follows that the fundamental group of $\bH^\sharp_{\rm der}$ embeds in the (trivial) fundamental group of $\bG^\sharp_{\rm der}$.
(This is observed in \cite[pp.~74-75]{MR0354892} and is easy to prove directly.)
It also follows
that our $z$-extension of $\bG$ restricts to a $z$-extension
$$1\to \bN \to \bH^\sharp\to \bH\to 1$$ of $\bH$.

Next, we observe that we have a natural ${\rm Gal}(E/F)$-equivariant identification of reduced buildings
$$\mathscr{B}_{\rm red}(\bH^\sharp ,E)\cong \mathscr{B}_{\rm red}(\bH,E)$$ as simplicial complexes.  If we identify the groups $$\bH^\sharp (E) /\bN (E) = (\bH^\sharp /\bN)(E) \cong \bH (E)$$ then $\bH (E)$ has the same action on $\mathscr{B}_{\rm red}(\bH^\sharp ,E)$ and $ \mathscr{B}_{\rm red}(\bH,E)$.

(Recall, from \cite[\S1.1]{MR546588}, that associated to each Chevalley basis of the Lie algebra $\fr{h}_{\rm der}$ of $\bH_{\rm der}$ is a point in $\mathscr{B}_{\rm red}(\bH,E)$.  But $\fr{h}_{\rm der}$ is naturally identified with the Lie algebra of $\bH^\sharp_{\rm der}$.
So a Chevalley basis determines a point in both $\mathscr{B}_{\rm red}(\bH,E)$
and $\mathscr{B}_{\rm red}(\bH^\sharp ,E)$.  Identifying these points determines our identification of reduced buildings.)

Given $x\in   \mathscr{B}_{\rm red}(\bH,E)$, we have an exact sequence
$$1\to \bN(E)\to \bH^\sharp(E)_x\to \bH (E)_x\to 1$$ of stabilizers of $x$, as well as exact sequences
$$1\to \bN(E)_r\to \bH^\sharp (E)_{x,r}\to \bH(E)_{x,r}\to 1$$ for all $r\ge 0$.
There are  similar exact sequences 
$$1\to N\to H^\sharp_x\to H_x\to 1$$ and $$1\to N_r\to H^\sharp_{x,r}\to H_{x,r}\to 1$$ 
for $F$-rational points.
(See \cite[Lemma 3.5.3]{KalYu}.)

\medskip
\begin{proposition}
Suppose $\rho$ is a representation of $H_x$ and $\rho^\sharp$ is its pullback to a representation of $H^\sharp_x$.  Then $\rho$ is permissible if and only if $\rho^\sharp$ is.  The representation $\pi^\sharp$ of $G^\sharp$ associated to $\rho^\sharp$ is precisely the pullback of the representation $\pi$ of $G$ associated to $\rho$.  The correspondence $\pi^\sharp\leftrightarrow \pi$ gives a bijection between the representations of $G$ that come from applying our construction to $G^\sharp$ and the representations that come from applying our construction directly to $G$.
\end{proposition}

To end this section, we note a complication that occurs for the group $\SL_n$.  Even though $\SL_n$ is simply connected, its dual ${\rm PGL}_n$ is not, and this introduces the issues discussed above in Remark \ref{conditionfourrem}.  Thus taking $z$-extensions does not remove all problems related to the residual characteristic.

\subsection{Commutators}

In this section, we collect a few facts about commutators.

Let $\bH (E)^+$ be the (normal) subgroup of $\bH (E)$ generated by the $E$-rational elements of the unipotent radicals of parabolic subgroups of $\bH$ that are defined over $E$. 

We recall now  some basic facts about $\bH (E)^+$.  (See  \cite{MR0164968} and \cite[\S7]{MR1278263}.)
Since  $E$ is a perfect field,   $\bH (E)^+$ may also be described as the group generated by the $E$-rational unipotent elements in $\bH (E)$.  
Since $E$ is a field with at least 4 elements,  any subgroup of $\bH (E)$ that is normalized by $\bH (E)^+$ either contains $\bH(E)^+$ or is central.  (This  follows from the main theorem of \cite{MR0164968}.)  
In particular, $\bH (E)^+\subset [\bH (E),\bH(E)]$.  But since $\bH$ is $E$-split, $\bH (E)/\bH(E)^+$ must be abelian and hence we deduce that
$$\bH (E)^+ = [\bH (E),\bH(E)].$$

\begin{lemma}\label{goodp}If $x\in \mathscr{B}_{\rm red}(\bH,F)$ then we have inclusions
\begin{eqnarray*}
H^\flat_{{\rm der},x,0+}&\subset& [H_{\rm der},H_{\rm der}]\cap H_{x,0+}\subset [H,H]\cap H_{x,0+}\\
&\subset& [\bH(E),\bH(E)]\cap H_{x,0+}
= \bH(E)^+ \cap H_{x,0+}\\
&\subset &H_{{\rm der},x,0+}.
\end{eqnarray*}
If the order of the fundamental group $\pi_1(\bH_{\rm der})$ is not divisible by $p$ then all of these inclusions are equalities.
\end{lemma}

\begin{proof}
The first step is to show  $H^\flat_{{\rm der},x,0+}\subset [H_{\rm der},H_{\rm der}]\cap H_{x,0+}$.
We observe that
the universal cover $\bH_{\rm sc}$ of $\bH_{\rm der}$ decomposes as a direct product
$$\bH_{\rm sc} = \prod_i \bH_i$$ of $F$-simple factors $\bH_i$, and,
for each $i$, we have $\bH_i = {\rm Res}_{F_i/F}\bH'_i$, where $\bH'_i$ is an absolutely simple, simply connected $F_i$-group and $F_i$ is a finite extension of $F$.  
(See \cite[\S3.1.2]{MR0224710}.)
To prove the desired inclusion, it suffices to show that for each $i$ we have $$\xi (H_{i,x,0+})\subset [\xi(H_i),\xi(H_i)],$$
where $\xi$ is the covering map $\bH_{\rm sc}\to \bH_{\rm der}$.

Suppose first that $\bH_i$ is $F$-isotropic or, equivalently, $\bH'_i$ is $F_i$-isotropic.  We have
$$H_i^+\subset [H_i,H_i]\subset H_i= H_i^+,$$ where the first inclusion follows from  \cite[Main Theorem]{MR0164968} and the equality follows from \cite[Theorem 7.6]{MR1278263}.  Therefore,
$$\xi (H_{i,x,0+}) \subset \xi([H_i,H_i]) \subset [\xi(H_i),\xi(H_i)].$$

Now suppose $\bH_i$ is $F$-anisotropic.   Then $H_i = \SL_1(D_i)$, where $D_i$ is a central division algebra over $F_i$.  According to \cite[Theorem 1.9]{MR1278263}, we have
$$[H_i,H_i]= H_i \cap (1+\mathfrak{P}_{D_i}) = H_{i,x,0+},$$ where $\mathfrak{P}_{D_i}$ is the maximal ideal in $D_i$.  Again, we obtain
   $\xi (H_{i,x,0+})\subset [\xi(H_i),\xi(H_i)]$ and hence $H^\flat_{{\rm der},x,0+}\subset [H_{\rm der},H_{\rm der}]\cap H_{x,0+}$.
   
The remaining inclusions follow from the inclusions
$$[H_{\rm der},H_{\rm der}]\subset [H,H]\subset [\bH(E),\bH(E)]
= \bH(E)^+ \subset\bH_{\rm der}(E).$$

Now suppose the order of the fundamental group $\pi_1(\bH_{\rm der})$ is not divisible by $p$.  The fact that $H^\flat_{{\rm der},x,0+}=  H_{{\rm der}, x,0+}$ is discussed in the proof of \cite[Lemma 3.5.1]{KalYu} and it follows from \cite[Lemma 3.1.1]{KalYu}.
 \end{proof}
 
\medskip
The previous proof uses standard methods that can also be found in \cite[\S7]{MR1278263} and in the proof of \cite[Lemma 3.5.1]{KalYu}.

There are several uses for the latter result.  
The fact that one often has $H^\flat_{{\rm der},x,0+}=  H_{{\rm der}, x,0+}$ allows one to replace $H^\flat_{{\rm der},x,0+}$ 
in the definition of ``permissible representation'' with a simpler object.   
The groups involving commutators are included partly to allow one to link 
our theory with Yu's theory, where quasicharacters of $H$ play more of a role.  In this regard, we note that a character of $H_{x,0+}$ extends to a quasicharacter of $H$ precisely when it is trivial on $[H,H]\cap H_{x,0+}$.  On the other hand,
the group $[H_{\rm der},H_{\rm der}]\cap H_{x,0+}$ is relevant to Kaletha's theory and, in particular, 
\cite[Lemma 3.5.1]{KalYu}.
Another application is the following:

\begin{proposition}
Suppose $\rho$ is a permissible representation of $H_x$ and $\chi$ is a quasicharacter of $G$.  Then $\rho\otimes (\chi|H_x)$ is also permissible.  The representation of $G$ associated to $\rho\otimes (\chi|H_x)$ is equivalent to $\pi \otimes \chi$, where $\pi$ is the representation of $G$ associated to $\rho$.
\end{proposition}

 \begin{proof}
Showing that $\rho\otimes (\chi|H_x)$ is  permissible amounts to showing that
$\chi | H^\flat_{{\rm der},x,0+}$ is trivial.  But this results from the inclusions
$$H^\flat_{{\rm der},x,0+}\subset [H,H]\subset [G,G],$$ which are implied by Lemma \ref{goodp}.  The second assertion follows directly from the definition of our construction.
 \end{proof}

\medskip
Next, we discuss commutators of groups associated to concave functions (as in the previous section) and, in particular, groups associated to admissible sequences.

Given two concave functions $f_1$ and $f_2$, a  function $f_1\vee f_2$ is defined in \cite[Definition B.1]{MR2431235} and it is shown in Lemma 5.22 
\cite{MR2431235} that, under modest restrictions, that $[G_{x,f_1}, G_{x,f_2}]\subset G_{x,f_1\vee f_2}$. (The latter result is a refinement of  \cite[Proposition 6.4.44]{MR0327923}.)

For groups associated to admissible sequences, Corollary 5.18 \cite{MR2431235} gives the following  simpler statement of the result in Lemma 5.22 \cite{MR2431235}.  (The notion of ``admissible sequence'' is more general in \cite{MR2431235} than in  \cite{MR1824988}.)

\bigskip
\begin{lemma}\label{AdlerSpice}
Given admissible sequences $\vec a = (a_0,\dots , a_d)$ and $\vec b = (b_0,\dots , b_d)$, 
then $[\vec G_{x,\vec a} , \vec G_{x,\vec b}]\subset \vec G_{x,\vec c}$, where 
$\vec c = (c_0,\dots , c_d)$, with 
$$c_j = 
\begin{cases}
\min\{ a_0+b_0,\dots , a_d+b_d\},&\text{if }j=0,\\ 
\min\{ a_j+m_b,m_a+b_j, a_{j+1}+b_{j+1},\dots , a_d+b_d\},&\text{if }j>0,
\end{cases}
$$
and $m_a = \min \{ a_0,\dots , a_d\}$ and $m_b = \min \{ b_0,\dots , b_d\}$.
 \end{lemma}

\medskip
\begin{remark}
We frequently will use the previous lemma (explicitly and implicitly) in combination with the following observation: to show that one subgroup $H_1$ of $G$ normalizes another subgroup $H_2$  is equivalent to showing $[H_1,H_2]\subset H_2$.  
\end{remark}

\subsection{Twists of depth zero representations}\label{sec:HisG}

The purpose of this section is to examine conditions under which a permissible representation $\rho$ of $H_x$ can be expressed as a tensor product $\rho_0\otimes \chi$, where $\rho_0$ has depth zero and $\chi$ is a quasicharacter of $H_x$.  Such a decomposition is useful in applications because the depth zero representations $\rho_0$ have a simple description, according to the work of Moy and Prasad \cite{MR1253198, MR1371680}.  A secondary goal is to contrast our construction of supercuspidal representations with Yu's construction  in the special case in which $\bG = \bH$ or, equivalently, $d=0$.

We start with a definition:

\begin{definition}
A quasicharacter $\chi$ of $H_x$ is {\bf $H$-normal} if it
satisfies one of the following two equivalent conditions:
 \begin{itemize}
\item $\chi$ is trivial on $[H,H_x]\cap H_x$.
\item The intertwining space 
 $I_h (\chi) =
{\rm Hom}_{
hH_{x} h^{-1}\cap H_{x}
} 
({}^h\chi ,\chi)
$ 
is nonzero for all $h\in H$.
\end{itemize}

\end{definition}

The next result and its corollary have obvious proofs:

\medskip
\begin{lemma}
If $\rho$ is an irreducible, smooth representation of $H_x$ and $\chi$ is  an $H$-normal quasicharacter  of $H_x$ then there is a set-theoretic identity $$I_h(\rho) = I_h(\rho\otimes \chi),$$ for all $h\in H$, where
$$I_h(\rho)={\rm Hom}_{hH_{x} h^{-1}\cap H_{x}}({}^h\rho ,\rho).$$
\end{lemma}

\medskip
\begin{corollary}\label{corinter}
If $\rho$ is an irreducible, smooth representation of $H_x$ and $\chi$ is  an $H$-normal quasicharacter  of $H_x$ then  the representation ${\rm ind}_{H_{x}}^H(\rho)$ is irreducible if and only if ${\rm ind}_{H_{x}}^H(\rho \otimes \chi)$ is irreducible.
\end{corollary}

The latter result raises the question of when, given a permissible representation $\rho$ of $H_x$, one can twist $\rho$ by an $H$-regular quasicharacter to obtain a depth zero representation, that is, a representation of $H_x$ whose restriction to $H_{x,0+}$ is a multiple of the trivial representation.
This is answered in the following:

\begin{lemma}
Suppose $\rho$ is a permissible representation of $H_x$ that restricts on $H_{x,0+}$ to the character $\phi$.  Then the following are equivalent:
\begin{itemize}
\item There exists an $H$-normal quasicharacter $\chi$ of $H_x$ and a depth zero representation $\rho_0$ of $H_x$ such that $\rho = \rho_0\otimes \chi$.
\item $\phi$ extends to an $H$-normal quasicharacter $\chi$ of $H_x$.
\item $\phi |([H,H_x]\cap H_{x,0+})=1$.
\end{itemize}
\end{lemma}

For Yu's tame supercuspidal representations, the conditions in the previous lemma can be replaced by the following more restrictive conditions:
\begin{itemize}
\item There exists a  quasicharacter $\chi$ of $H$ and a depth zero representation $\rho_0$ of $H_x$ such that $\rho = \rho_0\otimes (\chi |H_x)$.
\item $\phi$ extends to a  quasicharacter $\chi$ of $H$.
\item $\phi |([H,H]\cap H_{x,0+})=1$.
\end{itemize}

\bigskip
\begin{lemma}\label{goodptwo}
If $p$ does not divide the order of $\pi_1(\bH_{\rm der})$ then every permissible representation $\rho$ of $H_x$ may be expressed as $\rho_0\otimes (\chi|H_x)$, where $\rho_0$ is a depth zero representation of $H_x$ and $\chi$ is a quasicharacter of $H$.
\end{lemma}

\begin{proof}
Under the condition on $p$, Lemma \ref{goodp} implies that $H^\flat_{{\rm der},x,0+} = [H,H]\cap H_{x,0+}$.  Therefore the character $\phi$ of $H_{x,0+}$ associated to $\rho$ extends to a quasicharacter $\chi$ of $H$.  The representation $\rho_0 = \rho\otimes (\chi^{-1}|H_x)$ has depth zero.
\end{proof}

\medskip
\begin{remark}\label{newdequalszeroreps}
If $p$ divides the order of $\pi_1(\bH_{\rm der})$, then the possibility that $H^\flat_{{\rm der},x,0+}$ is strictly contained in $[H,H]\cap H_{x,0+}$
opens up the possibility of new supercuspidal representations  not captured by Yu's construction.
\end{remark}

We now compare our theory with Yu's theory in the case of $\bG = \bH$.  In this case, both theories minimally extend the Moy-Prasad theory of depth zero supercuspidal representations.

Let $\mathfrak{f}$ be the residue field of $F$, viewed as a subfield of the residue field $\mathfrak{F}$ of the maximal unramified extension of $F^{\rm un}$ of $F$ contained in $\overline{F}$.
Let $\mathsf{G}_x$ be the $\mathfrak{F}$-group with $\mathsf{G}_x (\mathfrak{F}) = \bG(F^{\rm un})_{x,0:0+}$ and the natural Galois action.  The group $\mathsf{G}_x (\mathfrak{f})$ is identified with $G_{x,0:0+}$.

Recall from \cite{MR1253198, MR1371680} that if $\rho^\circ$ is a representation of $G_{x,0}$ then
 $(G_{x,0},\rho^\circ)$ is called {\it a depth zero minimal $K$-type} 
 if it is the inflation of an irreducible cuspidal representation  of $G_{x,0:0+} = \mathsf{G}_x (\mathfrak{f})$.
 
 According to \cite[Proposition 6.8]{MR1371680} and \cite[Lemma 3.3]{MR1824988}, every irreducible, depth zero, supercuspidal representation $\pi$ of $G$ has the form ${\rm ind}_{G_{x}}^{G} (\rho_0)$, where $\rho_0$ is a smooth, irreducible representation of $G_{x}$ whose restriction to $G_{x,0+}$ is a multiple of the trivial representation and whose restriction to $G_{x,0}$ contains a depth zero minimal $K$-type.
 
 In fact,  every smooth, irreducible representation of $G_{x}$ that induces $\pi$ must contain a depth zero minimal $K$-type, since any two representations of $G_{x}$ that induce $\pi$ must be conjugate by an element of $G$ that normalizes $G_{x}$.

Yu's tame supercuspidal representations $\pi$ in the special case of $\bG =\bH$  are constructed as follows.    One starts with a pair $(\rho_0  ,\chi)$, where $\rho_0|G_{x,0+}$ is 1-isotypic,  $\rho_0$ induces an irreducible representation of $G$, and $\chi$ is any quasicharacter of $G$ and then one puts $$\pi = \ind_{G_x}^G (\rho_0\otimes (\chi |G_x)).$$
 Consider now the representation $\rho =\rho_0 \otimes (\chi|G_x)$.  In order for  $\rho$ to be  permissible, it must satisfy the following:
\begin{itemize}
\item[(1)] $\rho$ induces an irreducible representation of $G$,
\item[(2)] $\rho|G_{x,0+}$ is a multiple of some character $\phi$, 
\item[(3)] $\phi| H^\flat_{{\rm der},x,0+}=1$.
\end{itemize}

Condition (2) is satisfied with $\phi = \chi|G_{x,0+}$. 
Let $\chi^\sharp$ be the pullback of $\chi$ to a quasicharacter of a $z$-extension $G^\sharp$ of $G$.  Then Lemma 3.5.1 \cite{KalYu} implies that $\chi^\sharp$ is trivial on $G^\sharp_{{\rm der},x,0+}$.  Condition  (3) follows.  But the restriction of $\chi$ to $G_x$ is $G$-normal and hence condition (1) is satisfied by Corollary \ref{corinter}.  So $\rho$ must be permissible.

It is now evident  that when $\bG=\bH$  our framework captures all of Yu's supercuspidal representations for which $d=0$, and perhaps more in view of Remarks \ref{diszerorem} and \ref{newdequalszeroreps}.  

At first glance, it appears that  conditions (2) and (3) in the definition of ``permissible representation'' (Definition \ref{permissibledef}) should be eliminated since  they are either unnecessary or they reduce the number of supercuspidal representations constructed.  On the other hand, using condition (1) alone yields a theory with no content.
When $\bH$ is a proper subgroup of $\bG$, it will be readily apparent how all three conditions are needed to construct supercuspidal representations.

\subsection{Root space decompositions and exponential maps}\label{sec:expsection}

In this section, with $F$ as usual, $\bG$ can be taken as any connected reductive group that splits over $E$.  We also fix a point $x$ in the reduced building $\mathscr{B}_{\rm red}(\bG,F)$.

Let $\mathscr{M}$ consist of the following data:
\begin{itemize}
\item $\bT$ a maximal $F$-torus that splits over $E$ and contains $x$ in its apartment $\mathscr{A}_{\rm red}(\bG,\bT,E)$,
\item a $\Z$-basis $\chi_1,\dots , \chi_n$ for the character group $X^*(\bT)$,
\item a linear ordering of $\Phi = \Phi (\bG,\bT)$.
\end{itemize}


The theory of mock exponential maps was developed by Adler \cite[\S1.5]{MR1653184} and our approach is adapted from his, together with theory from \cite{MR0327923} and \cite{MR2508720}.

For positive $r$ in $$\widetilde\R = \R \cup \{ s+\, :\, s\in \R\}\cup \{ \infty\},$$
the mock exponential map on $\fr{t}(E)_r$,
$$\mathscr{M}\text{-exp}: \fr{t} (E)_{r}\to\bT(E)_{r},$$ is defined  by the condition
$$\chi_i (\mathscr{M}\text{-exp} (X)) = 1+d\chi_i (X),$$ for all $X\in \fr{t}(E)_{r}$ and  $1\le i\le n$.

Next, for each root $a\in \Phi$, we have an exponential map  
$$\text{exp}: \fr{g}_a (E)_{x,r}\to\bU_a(E)_{x,r}.$$

\begin{remark}
Recall that $\bU_a (E)_{x,r}$ is the group (denoted $X_\alpha$ in  \cite[\S1.1]{MR546588}) that Tits associates to the affine function $\alpha$ that has  gradient $a$ and satisfies $\alpha (x) =r$.  Thus $\bU_a(E)_{x,r} = \exp(E_{\alpha (*)}X_a)$, where $*$ is the point in $\mathscr{A}_{\rm red}(\bG,\bT,E)$ that Tits associates to some Chevalley basis that contains the root vector $X_a\in \fr{g}_a$.  Since  $\alpha (*)  = r-\langle x-*,a\rangle$, we have $$\bU_a (E)_{x,r}  = \exp (\fr{g}_a (E)_{x,r}),$$
where $$\fr{g}_a (E)_{x,r} = E_{r-\langle x-*,a\rangle} X_a.$$
\end{remark}

\bigskip
Now fix a function
$$f: \Phi \cup \{ 0\} \to (0,\infty )$$ that is concave in the sense that
$$f\left( \sum_i a_i\right) \le \sum_i f(a_i),$$ where we are summing over a set of elements of $\Phi\cup\{0\}$ whose sum is also in $\Phi \cup\{ 0\}$.

Then Bruhat-Tits consider the multiplication map
$$\mu : \bT(E)_{f(0)}\times \prod_{a\in \Phi}\bU_a(E)_{x,f(a)}\to \bG (E)_{x,f},$$
where the product is taken according some fixed (but arbitrary) ordering of $\Phi$, and $\bG(E)_{x,f}$ is the group generated by $\bT(E)_{f(0)}$ and the $\bU_a(E)_{x,f(a)}$'s.
It is shown in \cite[Proposition 6.4.48]{MR0327923} that $\mu$ is bijective.  (See also \cite[Lemma 5.22]{MR2431235}.)  We will consider the map $\mu$ associated to the ordering of $\Phi$ coming from $\mathscr{M}$.

\medskip
\begin{remark}
We refer the reader to \cite[\S2.25]{MR2508720} for more details on our application of \cite[Proposition 6.4.48]{MR0327923}.  Yu explains that the choice of a valuation of the root datum, in the sense of \cite{MR0327923}, is equivalent to the choice of a point in the reduced building.  In particular, it determines filtrations of the various root groups.  So it should be understood that when we apply the Bruhat-Tits result, we are choosing the valuation of root data corresponding to our given point $x$.
\end{remark}

The corresponding Lie algebra map 
$$\nu : \fr{t}(E)_{f(0)}\times \prod_{a\in \Phi}\fr{g}_a(E)_{x,f(a)}\to \fr{g} (E)_{x,f},$$
is obviously bijective and, in fact, is an $E$-linear isomorphism.

The mock exponential map
$$\mathscr{M}\text{-exp}: \fr{g} (E)_{x,f}\to\bG(E)_{x,f}$$
is defined as the composite
$$\xymatrix{ \fr{g} (E)_{x,f}\ar[r]^{\mathscr{M}\text{-exp}}\ar[d]_{\nu^{-1}}&\bG(E)_{x,f}\\
\fr{t}(E)_{f(0)}\times \prod_{a\in \Phi}\fr{g}_a(E)_{x,f(a)}\ar[r]&
\bT(E)_{f(0)}\times \prod_{a\in \Phi}\bU_a(E)_{x,f(a)}\ar[u]_{\mu}},$$
where the bottom map is obtained via the (mock) exponential maps discussed above on the factors.

\medskip
\begin{remark}
Once $\mathscr{M}$ and $x$ are fixed, the various mock exponential maps we get for different $f$ are all compatible in the sense that they are all restrictions of the map associated to the constant $0+$.
\end{remark}

The functions $f$ of most interest to us are the functions
$$f(a) = 
\begin{cases}
r_0,&\text{if }a\in \Phi (\bG^0,\bT)\cup\{0\},\\
r_i,&\text{if }a\in  \Phi (\bG^{i+1},\bT)- \Phi (\bG^i,\bT)\text{ and }i>0,
\end{cases}$$ associated to tamely ramified twisted Levi sequences $\vec\bG = (\bG^0,\dots ,\bG^d)$ and admissible sequences $\vec r = (r_0,\dots, r_d)$, as in \cite[\S1]{MR1824988}.  In this case, $\bG (E)_{x,f}$ is denoted  $\vec\bG (E)_{x,\vec r}$.  When $\vec r$ and $\vec s$ are admissible sequences such that
$$0<s_i\le r_i\le \min (s_i,\dots , s_d)+\min (s_0,\dots, s_d),$$ for all $i\in \{ 0,\dots , d\}$, then the mock exponential determines an isomorphism
$$\exp : \vec{\fr{g}}(E)_{x,\vec s:\vec r}\to \vec\bG(E)_{x,\vec s:\vec r}$$ of abelian groups that is independent of $\mathscr{M}$.  (See \cite[Lemma 1.3]{MR1824988}.)  The latter isomorphism is ${\rm Gal}(E/F)$-equivariant and yields an isomorphism
$\vec{\fr{g}}_{x,\vec s:\vec r}\cong  \vec G_{x,\vec s:\vec r}$ on $F$-points.  (See \cite[Corollary 2.4]{MR1824988}.)

\bigskip
\begin{remark}
The canonical reference for facts about groups associated to concave functions has historically been \cite{MR0327923}.  More concise and modern treatments of this theory can be found in \cite{MR2431235} and in the work of Yu.  Of particular importance to us are the results on commutators that are discussed in the next section.
\end{remark}

\bigskip
\begin{lemma}\label{HTtransfer}
Suppose $\chi$ is a character of $G_{x,0+}$ that is trivial on $\bG (E)^+\cap G_{x,0+}$.  Then the depth of $\chi$ is identical to the depth of $\chi |T_{0+}$.
\end{lemma}

\begin{proof}
Suppose $r>0$.  Then we have a commutative diagram
$$\xymatrix{ \fr{g} (E)_{x,r:r+}\ar[r]^{\mathscr{M}\text{-exp}}\ar[d]_{\nu^{-1}}&\bG(E)_{x,r:r+}\\
\fr{t}(E)_{r:r+}\times \prod_{a\in \Phi}\fr{g}_a(E)_{x,r:r+}\ar[r]&
\bT(E)_{r:r+}\times \prod_{a\in \Phi}\bU_a(E)_{x,r:r+}\ar[u]_{\mu}}$$
of isomorphisms.
We are especially interested in $\mu$, but its properties may be deduced from the properties of the other maps, which are studied in \cite{MR1824988}.
It is easy to see that $\mu^{-1}$ maps $G_{x,r:r+}$ to a set of the form $T_{r:r+}\times S$, where $S$ is a subgroup of $\prod_{a\in \Phi}\bU_a(E)_{x,r:r+}$ such that $\mu (S)$ is the image in $G_{x,r:r+}$ of a subset of $\bG(E)^+\cap G_{x,r}$.

Suppose $\chi$ is a character of $G_{x,0+}$ that is trivial on $G_{x,r+}$.  Then $\chi |G_{x,r}$ determines a character of $G_{x,r:r+}$ that factors through $\mu^{-1}$  to a character of $T_{r:r+}\times S$ that is the product of $\chi |T_{r:r+}$ with the trivial character of $S$.  It follows that if $\chi$ has depth $r$ then its restriction to $T_{0+}$ also has depth $r$.

Similarly, if $\chi$ is a character of $G_{x,0+}$ that is trivial on $T_{r+}$ and $\bG(E)^+\cap G_{x,0+}$ then $\chi$  determines a character of $G_{x,r:r+}$ and a corresponding character of $T_{r:r+}\times S$ that is trivial on $S$.
It follows that if $\chi|T_{0+}$ has depth $r$ then $\chi$ also has depth $r$.
\end{proof}

\subsection{The norm map and the filtrations of the $\bT_\mathscr{O}$'s}

Fix   an orbit $\mathscr{O}$ of $\Gamma = {\rm Gal} (\overline{F}/F)$ in $\Phi = \Phi (\bG,\bT)$,  and recall that we have let $\bT_{\mathscr{O}}$ denote the $F$-torus generated by the $E$-tori
$\bT_a = \text{image}(\check a)$ for $a\in \mathscr{O}$.

In this section, we study the Moy-Prasad filtration of the $T_{\mathscr{O}}$ using 
the norm map \begin{eqnarray*}
N_{E/F}: \bT_{\mathscr{O}}(E)&\to& T_{\mathscr{O}} \\
t&\mapsto& \prod_{\gamma\in {\rm Gal}(E/F)} \gamma (t).
\end{eqnarray*}

%
%
%


Our main result is:

\begin{lemma}\label{painintheneck}
If $a\in \mathscr{O}$ and $r>0$ then $$T_{\mathscr{O},r} =N_{E/F}(\check a (E_r^\times))= N_{E/F}(\bT_a(E)_r)= N_{E/F}(\bT_{\mathscr{O}}(E)_r)$$ for all $a\in\mathscr{O}$.
\end{lemma}

\begin{proof} 
Given a surjective homomorphism
$$\bA\to \bB$$ of $F$-tori that split over $E$
the associated homomorphism
$$A_r\to B_r$$
is surjective for all $r>0$, so long as the kernel of the original map is connected or finite with order not divisible by $p$.  (See Lemma 3.1.1 \cite{KalYu}.)

We can apply this to the norm map
$$N_{E/F}:R_{E/F}(\bT_{\mathscr{O}})\to \bT_{\mathscr{O}},$$ since it is surjective and its kernel is a torus.  (See the remarks below, following the proof.)
We deduce that
$$N_{E/F}(\bT_{\mathscr{O}}(E)_r) = T_{\mathscr{O},r}$$

Similarly, the surjection
$$\check a :R_{E/F}(\mathbb{G}_m)\to\bT_a$$ yields surjections
$$\bT_a (E)_r =\check a(E_r^\times )$$
for all $r>0$.
(If the kernel is nontrivial then it must have order two, in which case we invoke our assumption that $p\ne 2$.)

Since $\bT_{\mathscr{O}}(E)_r$ is generated by the groups $\bT_b (E)_r$ with $b\in \mathscr{O}$, it follows that $$N_{E/F}(\bT_a (E)_r) = N_{E/F}(\bT_{\mathscr{O}}(E)_r).$$
Our assertions now follow.
\end{proof}

\begin{remark}
The latter result asserts that, in a certain sense, the norm map ``preserves depth,''  but it does not  say that $N_{E/F}$ preserves the depths of individual elements.
This is similar to the fact that, since $E/F$ is tamely ramified,  
$\tr_{E/F}(E_r) = F_r$ for all real numbers $r$, even though the trace clearly does not preserve depths of individual elements.
\end{remark}

\medskip
\begin{remark}
It is not true in general that $\bT_a(E) = \check a (E^\times)$ for all roots $a$ that are defined over $E$.
Consider, for example, $\bG = {\rm PGL}_2$ and the subgroup $\bT$ consisting of cosets $g\bZ$, where $g\in \GL_2$ and $\bZ$ is the center of $\GL_2$.  Then $\check a(E^\times)$ consists of cosets 
$$\begin{pmatrix}u&0\\ 0&u^{-1}\end{pmatrix}\bZ = \begin{pmatrix}u^2&0\\ 0&1\end{pmatrix}\bZ,$$ with $u\in E^\times$, while $\bT_a(E)$ consists of cosets $$\begin{pmatrix}u&0\\ 0&1\end{pmatrix}\bZ,$$ with $u\in E^\times$.
\end{remark}

\medskip
\begin{remark}
Let $\bS$ be an $F$-torus.  If $L$ is a field containing $E$ then the norm map
$$N_{E/F}:R_{E/F}(\bS) \to \bS$$ over $L$ is equivalent to the multiplication map
$$\bS^n = \bS\times\cdots \times \bS\to \bS.$$  Accordingly, the norm map must be surjective and its kernel must be a torus.  If $\bS$ is $E$-split then the choice of an $E$-diagonalization of $\bS$ yields a realization of $\bS$ as a product of $R_{E/F} (\mathbb{G}_m)$ factors. Thus the norm gives an explicit realization of every $F$-torus as a quotient of an induced torus.  Similarly, the inclusion of $\bS$ in $R_{E/F}(\bS)$ explicitly shows that every $F$-torus is contained in an induced torus.
(See Propositions 2.1 and 2.2 \cite{MR1278263}.)
\end{remark}

\subsection{Yu's genericity conditions}\label{sec:GE2}

This section is a recapitulation of facts from \cite{MR1824988} that are, implicitly or explicitly, based on results from  \cite{MR0354892}.

Fix $i\in \{ 0,\dots , d-1\}$ and suppose $X^*_i$ lies in the dual coset $(\phi |Z^{i,i+1}_{r_i})^*\in \fr{z}^{i,i+1,*}_{-r_i: (-r_i)+}$.

For each root $a\in \Phi^{i+1} = \Phi (\bG^{i+1},\bT)$, let $H_a = d\check a(1)$  be the associated coroot in the Lie algebra $\fr{g}^{i+1}$.

In Lemma \ref{GEonelemma}, we will show that $X^*_i$ necessarily satisfies Yu's condition (\cite[p.~596]{MR1824988}):

\bigskip
{\bf GE1.}   $v_F(X^*_i(H_a))=-r_i,\quad \forall a\in \Phi^{i+1}-\Phi^i$.

\bigskip
For the rest of this section, we assume {\bf GE1} is satisfied.

Now let $\fr{F}$ be the residue field of the algebraic closure $\overline{F}$ of $F$ or, equivalently, the algebraic closure of the residue field $\fr{f}$.

Choose $\varpi_{r_i}\in \overline{F}$ of depth $r_i$ and note that $\varpi_{r_i}X^*_i$ has depth 0.

Given $\chi\in X^*(\bT)$, we have an associated linear form $d\chi :\bfr{t}\to \mathbb{G}_a$, and the map $\chi\mapsto d\chi$ extends to an isomorphism $$X^*(\bT)\otimes_\Z \overline{F} \cong \bfr{t}^*\otimes \overline{F}.$$
Under the latter isomorphism, the element $\varpi_{r_i}X^*_i$ corresponds to an element of $X^*(\bT)\otimes_\Z \fr{O}_{\overline F}$.

Let $\widetilde{X}^*_i$ be the residue class of $\varpi_{r_i}X^*_i$ modulo
$X^*(\bT)\otimes_\Z \fr{P}_{\overline F}$.
Thus $$\widetilde{X}^*_i\in X^*(\bT)\otimes_\Z \fr{F} .$$
Though $\widetilde X^*_i$ is only well-defined up to multiplication by a scalar in $\fr{F}^\times$ (depending on the choice of $\varpi_{r_i}$), it is easy to see that this scalar does not affect our discussion.  (In other words, one could treat $\widetilde{X}^*_i$ as a projective point.)

Let $$\Phi^{i+1}_{\widetilde{X}_i^*}= \left\{ a\in \Phi^{i+1}\ :\ \widetilde X_i(H_a)=0\right\}.$$

\medskip
\begin{lemma}
$\Phi^{i+1}_{\widetilde{X}_i^*}= \Phi^i$.
\end{lemma}

\begin{proof}
{\bf GE1} implies that $\widetilde{X}^*_i (H_a) \ne 0$ for all $a\in \Phi^{i+1}-\Phi^i$.  Since $X^*_i\in \fr{z}^{i,*}$ and $H_a\in \fr{g}_{\rm der}^i$ for $a\in \Phi^i$, we must have $\widetilde{X}^*_i (H_a)=0$ for all $a\in \Phi^i$.  
\end{proof}

We now observe that the Weyl group 
$$W(\Phi^{i+1}) = \langle r_a\ :\ a\in \Phi^{i+1}\rangle$$ acts on
$X^*(\bT)\otimes_\Z\fr{F}$ and we let $Z_{W(\Phi^{i+1})}(\widetilde X^*_i)$ denote the stabilizer of $\widetilde X^*_i$.

Yu's second genericity condition (\cite[p.~596]{MR1824988}) is:

\bigskip
{\bf GE2.}   $Z_{W(\Phi^{i+1})}(\widetilde X^*_i) =W (\Phi^i) = W\!\left(\Phi^{i+1}_{\widetilde X_i}\right)$.

\bigskip
\begin{lemma}\label{Steinberglemma}
$W(\Phi^i) = W\!\left(\Phi^{i+1}_{\widetilde X_i}\right)$ is identical to the subgroup of
$Z_{W(\Phi^{i+1})}(\widetilde X^*_i)$ generated by the reflections in $Z_{W(\Phi^{i+1})}(\widetilde X^*_i)$.
\end{lemma}

\begin{proof}
If $a\in \Phi^{i+1}$ then
$$r_a\left(\widetilde{X}^*_i \right) = \widetilde X^*_i - \widetilde X^*_i (H_a)\, a,$$
 and thus $r_a \in Z_{W(\Phi^{i+1})}(\widetilde X^*_i)$ precisely when $a\in \Phi^{i+1}_{\widetilde{X}^*_i} = \Phi^i$.
Since every reflection in $W(\Phi^{i+1})$ has the form $r_a$ for some $a\in \Phi^{i+1}$, our claim follows.
\end{proof}

\begin{lemma}
$W(\Phi^i)$ is a normal subgroup of $Z_{W(\Phi^{i+1})}(\widetilde X^*_i)$ and the quotient group
$Z_{W(\Phi^{i+1})}(\widetilde X^*_i)/W(\Phi^i)$ is isomorphic to a subgroup of the torsion component of $X^*(\bT)/\Z\Phi^{i+1}$, where $\Z\Phi^{i+1}$ is the lattice generated by $\Phi^{i+1}$. 
\end{lemma}

\begin{proof}
The desired result follows directly from Theorem 4.2(a)  in \cite{MR0354892} upon  equating Steinberg's objects
$$(H,A,L,{}^AL,\Sigma^*,W, Z_W(H))$$
with the objects 
$$(\widetilde X^*_i, \fr{F}, X^*(\bT), X^*(\bT)\otimes_\Z \fr{F},\Phi^{i+1}, W(\Phi^{i+1}),Z_{W (\Phi^{i+1})} (\widetilde X^*_i)),$$
respectively.
Note that $Z_W(H)^0$ in \cite{MR0354892} denotes the subgroup of $Z_W(H)$ generated by the reflections in $Z_W(H)$.\end{proof}

\medskip
\begin{corollary} {\rm (\cite[Lemma 8.1]{MR1824988})}
If $p$ is not a torsion prime for the dual of the root datum of $\bG^{i+1}$, that is, $(X_*(\bT),\check\Phi^{i+1}, X^*(\bT),\Phi^{i+1})$ then   {\bf GE1} implies {\bf GE2}.  
\end{corollary}

\subsection{Genericity of the dual coset}

This section involves the dual cosets defined in Definition \ref{dualcosetdef}.
The  result we prove is inspired by Lemma 3.7.5 \cite{KalYu}.

\begin{lemma}\label{GEonelemma}
If  $i\in \{ 0,\dots , d-1\}$ then every  element $X^*_i$ in the dual coset $(\phi |Z^{i,i+1}_{r_i})^*$ consists of $\bG^{i+1}$-generic elements of depth $-r_i$.\end{lemma}

\begin{proof}
Given $a\in \Phi^{i+1}-\Phi^i$, let $H_a= d\check a (1)$.
Let $\mathscr{O}$ be the $\Gamma$-orbit of $a$.
With the obvious notations, we have the Lie algebra analogue
of Lemma \ref{painintheneck}:
\begin{eqnarray*}
\fr{t}_{\mathscr{O},r_i}&=& \tr_{E/F}(d\check a (E_{r_i})) = \tr_{E/F}(E_{r_i}H_a)\\ &=& \tr_{E/F}(\fr{t}_a(E)_{r_i})= \tr_{E/F}(\fr{t}_{\mathscr{O}}(E)_{r_i}).\end{eqnarray*}
(The proof is analogous to that of Lemma \ref{painintheneck}.)

For $u\in E_{r_i}$, we therefore have 
$\tr_{E/F}(uH_a) \in \fr{t}_{\mathscr{O},r_i}\subset \fr{t}^i_{r_i}\subset 
\fr{h}^{i}_{x, r_i}$ and
\begin{eqnarray*}
\psi(\tr_{E/F} (u X^*_i(H_a)))
&=&\psi(X^*_i(\tr_{E/F} (uH_a)))\\
&=&\phi (\exp (\tr_{E/F}(uH_a))) \\
&=&\phi (N_{E/F}(\exp (uH_a))) \\
&=&\phi (N_{E/F}(\exp (d\check a(u)))) \\
&=&\phi (N_{E/F}(\check a(1+u))) .
\end{eqnarray*}
Here, $\exp$ is the Moy-Prasad isomorphism $\fr{h}^{i}_{x, r_i:r_i+}\cong H^{i}_{x, r_i:r_i+}$, and we observe that $\exp$ restricts to the Moy-Prasad isomorphism $\fr{t}^i_{r_i:r_i+}\cong T^{i}_{r_i:r_i+}$.

We now observe  that $\phi |T_{\mathscr{O},r_i}\ne 1$ for all $\mathscr{O}\in \Gamma\bs \Phi^{i+1}$ and, according to Lemma \ref{painintheneck}, we also know that if $\mathscr{O}$ is the orbit of $a$ then $N_{E/F}(\check a (E^\times_{r_i})) = T_{\mathscr{O},r_i}$.  Therefore, there exists $u\in E_{r_i}$ such that
$\psi(\tr_{E/F} (u X^*_i(H_a)))\ne 1$.  
This implies $v_F (X^*_i (H_a))=-r_i$, which shows that condition {\bf GE1} of \cite{MR1824988} is satisfied.

Now we invoke Lemma 8.1 \cite{MR1824988} or condition (4) in Definition \ref{permissibledef} to conclude that Yu's condition {\bf GE2} is also satisfied.
\end{proof}

\subsection{Heisenberg $p$-groups and finite Weil representations}\label{sec:HeisWeil}

Fix $i\in \{ 0,\dots ,d-1\}$ and let $\mu_p$ be the group of complex $p$-th roots of unity.
We have a surjective homomorphism $$\zeta_i : J^{i+1}_+ \to \mu_p$$ which we have used to define a symplectic form on the space
$$W_i= J^{i+1}/J^{i+1}_+$$
by $$\langle uJ^{i+1}_+,vJ^{i+1}_+\rangle = \zeta_i (uvu^{-1}v^{-1}).$$
 We have also defined 
a (multiplicative) Heisenberg $p$-group structure on 
$$\mathscr{H}_i = W_i\times \mu_p$$
using the multiplication rule
$$(w_1,z_1)(w_2,z_2) = (w_1w_2, z_1z_2 \langle w_1,w_2\rangle^{(p+1)/2}).$$
The center of $\mathscr{H}_i$ is  $$\mathscr{Z}_i = 1\times \mu_p$$  and we use $z\mapsto (1,z)$ to identify $\mu_p$ with $\mathscr{Z}_i$.

We let $H_{x}$ act on $\mathscr{H}_i$ by
$$h\cdot (jJ_+^{i+1},z) = (hjh^{-1}J_+^{i+1},z).$$

We are interested in (central) group extensions of $W_i$ by $\mu_p$ that are isomorphic to $\mathscr{H}_i$ (as group extensions).  To say that an extension
$$1\to\mu_p\to \widetilde W_i\to W_i\to 1$$
of $W_i$ by $\mu_p$ is isomorphic to $\mathscr{H}_i$ as a group extension means that there exists an isomorphism $\iota  :\widetilde{W}_i\to \mathscr{H}_i$ such that the diagram 
$$\xymatrix{1\ar@{=}[d]\ar[r]&\mu_p\ar[r]\ar@{=}[d]&\widetilde{W}_i\ar[r]\ar[d]^\iota&W_i\ar[r]\ar@{=}[d]&1\ar@{=}[d]\\
1\ar[r]&\mathscr{Z}_i\ar[r]&\mathscr{H}_i\ar[r]&W_i\ar[r]&1}$$
commutes.  Such an isomorphism $\iota$ is called a {\it special isomorphism}.

(Note that, by the Five Lemma, if we merely assume $\iota$ is a homomorphism making the latter diagram commute then it is automatically an isomorphism.)

Recall that the cohomology classes in $H^2(W_i,\mu_p)$ correspond to group extensions 
$ \widetilde W_i$
of $W_i$ by $\mu_p$.
Given $\widetilde W_i$, we get a 2-cocycle  $$c: W_i\times W_i\to \mu_p$$ by choosing a section $s:W_i\to \widetilde W_i$ and defining $c$ by
$$s(w_1)s(w_2) = c(w_1,w_2)\, s(w_1w_2).$$
To say that $c$ is a 2-cocycle means that
$$c(w_1,w_2)\ c(w_1w_2,w_3) = (s(w_1)c(w_2,w_3)s(w_1)^{-1})\, c(w_1,w_2w_3).$$
Indeed, both sides of the latter equation equal
$$s(w_1)s(w_2)s(w_3)s(w_1w_2w_3)^{-1}.$$

Changing the section $s$ has the effect of modifying $c$ by a 2-coboundary.
Indeed, any other section $s'$ must be related to $s$ by $$s' (w) = f(w)s(w)$$ where $f$ is a function $f:W_i\to \mu_p$, and then $c' = c\cdot df$, where
$$(df)(w_1,w_2) =f(w_1) f(w_2)f(w_1w_2)^{-1}.$$

On the other hand, given a 2-cocycle $c$ we get an extension  $\widetilde{W}_i$ of $W_i$ by $\mu_p$ by taking $\widetilde{W}_i = W_i\times\mu_p$ with multiplication defined by 
$$(w_1,z_1)(w_2,z_2) = (w_1w_2, c (w_1,w_2) \, z_1z_2)$$ for some coycle
$c : W_i\times W_i\to \mu_p$.  
Associativity of multiplication is equivalent to the cocycle condition.

The cocycle
$$c(w_1,w_2) = \langle w_1,w_2\rangle^{(p+1)/2}$$
yields the Heisenberg $p$-group $\mathscr{H}_i$ and we are only interested in cocycles that are cohomologous to this cocycle.
In fact, we are only interested in special isomorphisms that come from  homomorphisms $\nu_i :J^{i+1}\to \mathscr{H}_i$.

\begin{definition}\label{spechom}
A {\it special homomorphism} is an $H_{x}$-equivariant homomorphism $\nu_i : J^{i+1}  \to \mathscr{H}_i$  that factors to a special isomorphism $J^{i+1}/\ker\zeta_i \to \mathscr{H}_i$.
\end{definition}

We observe that when $\nu_i$ is a  special homomorphism then
$\ker\nu_i = \ker\zeta_i$ and
the diagram 
$$\xymatrix{1\ar@{=}[d]\ar[r]&J^{i+1}_+/\ker\zeta_i\ar[r]\ar[d]^{\zeta_i}&J^{i+1}/\ker\zeta_i\ar[r]\ar[d]^{\nu_i}&W_i\ar[r]\ar@{=}[d]&1\ar@{=}[d]\\
1\ar[r]&\mathscr{Z}_i\ar[r]&\mathscr{H}_i\ar[r]&W_i\ar[r]&1}$$
commutes.

\bigskip\begin{remark}
Our notion of special homomorphism is a variant of the notion of ``relevant special isomorphism'' in \cite[Definition 3.17]{MR2431732} which, in turn, is derived from Yu's definition of ``special isomorphism''  \cite[\S10]{MR1824988}.  The existence of special homomorphisms follows from the existence of relevant special isomorphisms.  (See \cite[Lemma 3.22, Proposition 3.24]{MR2431732}.)
\end{remark}

\bigskip\begin{remark}
Both $\mathscr{H}_i$ and $J^{i+1}/\ker \zeta_i$ are Heisenberg $p$-groups that are canonically associated to $\rho$.   Special homomorphisms yield isomorphisms between these two Heisenberg $p$-groups, but the existence of automorphisms of the Heisenberg $p$-groups leads to the lack of uniqueness of special isomorphisms.  One can choose canonical special homomorphisms, however, there is no one canonical choice that is most convenient for all applications.  (These matters are discussed in \cite[\S3.3]{MR2431732}.)
\end{remark}

\bigskip
We have chosen a Heisenberg representation $(\tau_i ,V_i)$ of $\mathscr{H}_i$ and we have pulled back $\tau_i$ to a representation $(\tau_{\nu_i},V_i)$ of $J^{i+1}$.  Thus $\tau_{\nu_i} =\tau_i\circ \nu_i$.
We have extended $\tau_i$ to a representation $\hat\tau_i$ of $\mathscr{S}_i\ltimes \mathscr{H}_i$ on $V_i$ and we have let $\omega_i (h) = \hat\tau_i ({\rm Int}(h),1)$, for all $h\in H_{x}$.

The next lemma follows routinely from the definitions, but we include a proof to make evident where the $H_{x}$-equivariance of $\nu_i$ is used.

\begin{lemma}\label{taunuiidentity}
$$\tau_{\nu_i}(h^{-1}jh) = \omega_i(h)^{-1}\tau_{\nu_i}(j)\, \omega_i (h),$$ whenever $j\in J^{i+1}$ and $h\in H_{x}$.
\end{lemma}

\begin{proof} Given $j\in J^{i+1}$ and $h\in H_{x}$, our assertion follows from the computation
\begin{eqnarray*}
\tau_{\nu_i}(h^{-1}jh)
&=&\hat\tau_i (1,\nu_i (h^{-1}jh))\\
&=&\hat\tau_i (1,{\rm Int}(h)^{-1}\nu_i(j))\\
&=& \omega_i(h)^{-1}\tau_{\nu_i}(j)\, \omega_i(h).
\end{eqnarray*}
\end{proof}

\subsection{Verifying that $\hat\phi$ is well defined}\label{sec:hatphiwelldef}

Recall from \S\ref{sec:construction} the following definitions:
\begin{eqnarray*}
 L^{i+1}_+&=& G^{i+1}_{\rm der}\cap J^{i+1}_{++},\\
L_+&=& (\ker \phi)L^1_+\cdots L^d_+,\\
K_+&=& H_{x,0+}L_+,\\
\bH^{i} &=& \bG^{i+1}_{\rm der}\cap \bH.
\end{eqnarray*}
We have also stated a provisional definition for the character $\hat\phi$ of $K_+$, namely, $\hat\phi = {\rm inf}_{H_{x,0+}}^{K_+}\kern-.3em (\phi )$  or, equivalently, $\hat\phi (h\ell)= \phi( h)$, when $h\in H_{x,0+}$ and $\ell\in L_+$.  

The  purpose of this section is to verify  that the definition of $\hat\phi$ make sense.

\begin{lemma}\label{Kplusidentity}
\begin{itemize}
\item[{\rm (1)}] $L_+$ is a group that is normalized by $H_{x,0+}$.
\item[{\rm (2)}] $K_+ = \vec{G}_{x,(0+,s_0+,\dots, s_{d-1}+)}= H_{x,0+}J_+$.
\item[{\rm (3)}] $H \cap L_+ = \ker \phi$.
\item[{\rm (4)}] $\hat\phi$ is a well-defined character of $K_+$.
\item[{\rm (5)}] $\hat\phi (h) =\zeta_i (h)$, when $h\in G^{i+1}_{{\rm der}}\cap J^{i+1}_+$ and $i\in \{ 0,\dots , d-1\}$.
\end{itemize}
\end{lemma}

\bigskip
\begin{remark}
Our definition of $K_+$ is different than the definition  in \cite{MR1824988} and \cite{MR2431732}, however, Lemma \ref{Kplusidentity} shows that the different definitions coincide.  The notation $\vec{G}_{x,(0+,s_0+,\dots, s_{d-1}+)}$ is a special case of  the notation from \S\ref{sec:expsection} for groups associated to admissible sequences.
\end{remark}

\begin{proof}[Proof of Lemma \ref{Kplusidentity}]
By definition,  $L_+ =(\ker \phi)L^1_+\cdots L^d_+$ and since $H_{x,0+}$ normalizes each factor $L^j_+$, so does $\ker \phi$.
Moreover, $L^{j_1}_+$ normalizes $L^{j_2}_+$ whenever $j_1\le j_2$.  (See Remark \ref{AdlerSpice}.)
It follows from induction on $i$ that the sets $$(\ker \phi)L^1_+\cdots L^i_+$$ are groups, for $i = 1,\dots ,d$.
This yields (1).

Now suppose $i\in \{ 0,\dots , d-1\}$ and let $$\bG^{i,i+1}(E)_{x,s_i+}=\prod_{a\in \Phi^{i+1}-\Phi^i} {\bf U}_a (E)_{x,s_i+},$$ where ${\bf U}_a$ is the root group associated to the root $a$. In the latter definition, we assume we have fixed an ordering of the set $\Phi^{i+1} -\Phi^i$ and we use this ordering to determine the order of multiplication in the product.  
The resulting set depends on this ordering.


We observe that
$$\bG^{i,i+1}(E)_{x,s_i+}\subseteq \bG^{i+1}_{\rm der}(E)\cap (\bG^i,\bG^{i+1})(E)_{x,(r_i+,s_i+)}\subseteq \bG^{i+1}(E)_{x,s_i+}$$
and, according to \cite[Lemma 1.4]{MR1824988},
$$\vec{\bG}(E)_{x,(0+,s_0+,\dots, s_{d-1}+)}
= \bH (E)_{x,0+}\bG (E)^{1}_{x,s_0+}\cdots \bG (E)^{d}_{x,s_{d-1}+},$$ 
The latter identity may be sharpened as
$$\vec{\bG}(E)_{x,(0+,s_0+,\dots, s_{d-1}+)}
= \bH (E)_{x,0+}\bG (E)^{0,1}_{x,s_0+}\cdots \bG (E)^{d-1,d}_{x,s_{d-1}+},$$ where we are using \cite[Lemma 1.4]{MR1824988}
and 
 \cite[Proposition 6.4.48]{MR0327923}.
 
 More precisely, the Bruhat-Tits result can be used to show that one can rearrange the various products so that the contributions of the individual root groups occur in any  specified order.  In particular, one can first group together the roots from $\Phi^0$, then the roots from $\Phi^1 - \Phi^0$, and so forth.

 It follows that
 $$\vec{\bG}(E)_{x,(0+,s_0+,\dots, s_{d-1}+)}\!\!
=\! \bH (E)_{x,0+}\!\prod_{i=0}^{d-1}\!\! \left(\! \bG^{i+1}_{\rm der}(E)\! \cap \! (\bG^i,\bG^{i+1})(E)_{x,(r_i+,s_i+)}\right)\! .$$

Using the Bruhat-Tits result or an argument as in the proof of \cite[Lemma 2.10]{MR1824988}, we have:
\begin{eqnarray*}
\vec{G}_{x,(0+,s_0+,\dots, s_{d-1}+)}
&=& H_{x,0+}\!\prod_{i=0}^{d-1} L^{i+1}_+\\
&=& H_{x,0+}L_+ = K_+.
\end{eqnarray*}
The identity $$\vec{G}_{x,(0+,s_0+,\dots, s_{d-1}+)}= H_{x,0+}J_+$$ is established in \cite{MR1824988}.  Thus we have proven (2).

Assertion (3) follows from:
$$H\cap L^{i+1}_+ = G^{i+1}_{\rm der}\cap H_{x,r_i+} = H^i_{x,r_i+}\subset H^i_{x,r_{i+1}}\subset \ker \phi .$$

We have defined $\hat\phi$ by $\hat\phi (h\ell) = \phi (h)$, for $h\in H_{x,0+}$ and $\ell\in L_+$.  The fact  that this is a well-defined function on $K_+$ is a consequence of the fact that $H_{x,0+}\cap L_+ =\ker \phi$.
We now use (1) and the computation
\begin{eqnarray*}
\hat\phi (h_1\ell_1h_2\ell_2)&=&
\hat\phi (h_1h_2(h_2^{-1}\ell_1 h_2)\ell_2)\\
&=&\phi(h_1h_2)\\
&=&\hat\phi (h_1\ell_1)\hat\phi(h_2\ell_2),
\end{eqnarray*}
for $h_1,h_2\in H_{x,0+}$ and $\ell_1,\ell_2\in L_+$
to deduce that $\hat\phi$ is a character and prove (4).

Finally, we prove (5).  The methods used earlier in this proof can be used (with essentially no modification) to show that 
$$G^{i+1}_{\rm der}\cap J^{i+1}_+ = H^i_{x,r_i} L^1_+\cdots L^{i+1}_+.$$
We will show that $\hat\phi$ and $\zeta_i$ agree on each of the factors on the right hand side of the latter identity.

On the factor $H^i_{x,r_i}$, the characters $\hat\phi$ and $\zeta_i$ coincide since both are represented by any element ${X}^*_i$ in the dual coset $(\phi |Z^{i,i+1}_{r_i})^*$ defined in \S\ref{sec:dualcoset}.

Now consider a factor $L^{j+1}_+$, with $j\in \{ 0,\dots , i-1\}$.  By definition,  $\hat\phi$ is trivial on such a factor.  The fact that $\zeta_i$ is also trivial on $L^{j+1}_+$ follows from the fact that the dual coset $(\phi |Z^{i,i+1}_{r_i})^*$ is contained in $(\mathfrak{z}^{i,*})_{-r_i}$ which is orthogonal to $\mathfrak{g}^i_{{\rm der}, x,r_i}$.  Here, we are using the fact that $L^{j+1}_+ \subset G^i_{\rm der}\cap G^i_{x,r_i}$.  We also know that both $\hat\phi$ and $\zeta_i$ are trivial on the factor $L^{i+1}_+$.  This completes our proof.
\end{proof}

\subsection{Intertwining theory for Heisenberg $p$-groups}

The following section was prompted by an email message from Loren Spice and it was developed in the course of conversations with him.  It addresses an error in the proofs of Proposition 14.1 and Theorem 14.2 of \cite{MR1824988} that can be fixed using ingredients already present in \cite{MR1824988}.  

Roughly speaking, our point of view on intertwining is that two representations of overlapping groups intertwine when they can be glued together to form a larger representation of a larger group.  This approach reveals more of the inherent geometric structure present in Yu's theory.

As usual, $\bG$  will be a connected reductive $F$-group, but now $\bG'$ will be an $F$-subgroup  of $\bG$ that is a Levi subgroup over $E$.  In other words, $(\bG',\bG)$ is a twisted Levi sequence in $\bG$.  Let $x$ be a point in $\mathscr{B}_{\rm red}(\bG',F)$.

As in \cite[\S9]{MR1824988}, we put $J= (G',G)_{x,(r,s)}$, $J_+= (G',G)_{x,(r,s+)}$
and we let $\phi$ be a $\bG$-generic character of $G'_{x,r}$ of depth $r$.  We let $\zeta$ denote the  character   of $J_+$ associated to $\phi$.  (In other words, $\zeta$ is the character denoted by $\hat\phi |J_+$ in \cite{MR1824988}.)

We remark that the quotient $W= J/J_+$ is canonically isomorphic to the corresponding Lie algebra quotient and thus our discussion could be carried out on the Lie algebra.

For subquotients $S$ of $G$, let ${}^g S$ denote the subquotient obtained by conjugating by $g$.
For a representation $\pi$ of a subquotient $S$ of $G$, let ${}^g\pi$ be the representation of ${}^g S$ given by ${}^g \pi (s) = \pi (g^{-1}sg)$.

Let $\Phi = \Phi (\bG,\bT)$, $\Phi' = \Phi (\bG',\bT)$ and 
let $\Phi_0$ be the set of roots $a\in \Phi$ such that $a(gx-x)=0$.

We define groups $J_0$ and $J_{0,+}$ by imitating the definitions of $J$ and $J_+$ except that we only use roots that lie in $\Phi_0$.

More precisely, let $J_0 (E)$ be the group generated by $\bT(E)_r$ and the groups $\bU_a(E)_{x,r}$, with $a\in \Phi'\cap \Phi_0$, and the groups $\bU_a(E)_{x,s}$, with $a\in (\Phi -\Phi')\cap \Phi_0$.  Define $J(E)_{0,+}$ similarly, but replace $s$ by $s+$.  Now let $J_0 = J(E)_0\cap G$ and $J_{0,+}= J(E)_{0,+}\cap G$.

Let $\mu_p$ be the group of complex $p$-th roots of unity.  Note that all of our symplectic forms are multiplicative and take values in $\mu_p$.  We let $I_p$ denote the identity map on $\mu_p$ viewed as a (nontrivial) character of $\mu_p$.

Define $W_0$ to be the space $J_0/J_{0,+}$, viewed as a (multiplicative) $\F_p$-vector space with $\mu_p$-valued symplectic form given by 
$$\langle uJ_{0,+},vJ_{0,+}\rangle = \zeta (uvu^{-1}v^{-1} )= {}^g \zeta (uvu^{-1}v^{-1} ).$$
(The latter equality follows from \cite[Lemma 9.3]{MR1824988}.)

The embedding $J_0\hookrightarrow J$ yields an embedding $W_0\hookrightarrow W$ such that the symplectic form on $W$, given by
$$\langle uJ_{+},vJ_{+}\rangle = \zeta (uvu^{-1}v^{-1} ),$$
 restricts to the symplectic form on $W_0$.
 
Similarly, the embedding ${}^g J_0\hookrightarrow {}^g J$ yields an embedding $W_0\hookrightarrow {}^gW$ such that the symplectic form on ${}^gW$, given by
$$\langle u\ {}^gJ_{+},v\ {}^g J_{+}\rangle = {}^g\zeta (uvu^{-1}v^{-1} ),$$
 restricts to the symplectic form on $W_0$.
 
 We therefore have a fibered sum (a.k.a., pushout) diagram 
$$\xymatrix{
W_0\ar[r]\ar[d]&{}^gW\ar[d]\\
W\ar[r]&W^\star .
}$$
The fibered sum $$W^\star = W \sqcup_{W_0} {}^gW$$ is
the quotient of $W\times {}^g W$ by the subgroup of elements $(w_0,w_0^{-1})$, where $w_0$ varies over $W_0$.

Let $W_1$ denote the image of ${}^gJ_+\cap J$ in $W$.  Then, according to \cite[Remark 12.7]{MR1824988}, the space $W_1$ is totally isotropic and its  orthogonal complement $W_1^\perp$ in $W$ is identical to the image of ${}^gJ\cap J$ in $W$.  We have $W_1^\perp = W_1\oplus W_0$.

As with $W_0$, Yu (canonically) defines a subgroup $J_1$ such that $W_1 = J_1J_+/J_+$, except that he uses the roots $a$ such that $a(x-gx)<0$.
He also defines a subgroup $J_3$ by using the roots $a$ such that $a(x-gx)>0$.  Taking $W_3 = J_3J_+/J_+$, gives a vector space such that we have an orthogonal direct sum decomposition $W = W_{13}\oplus W_0$, where the nondegenerate space $W_{13}$ has polarization $W_1\oplus W_3$.

We also have canonical identifications $$W= W_{13}\oplus W_0$$  and $${}^gW = W_0\oplus {}^g W_{13}.$$   (See Lemmas 12.8 and 13.6 \cite{MR1824988}.)
Since the symplectic forms that $W_0$ inherits from $W$ and ${}^gW$ coincide, the space $W^\star$ inherits a nondegenerate symplectic structure such that 
$$W^\star = W_{13}\oplus W_0\oplus {}^gW_{13}$$ is an orthogonal direct sum of nondegenerate symplectic spaces.

%
%


Let $$\mathscr{H}^\star = W^\star \times \mu_p$$ with the multiplication
$$(w_1,z_1)(w_2,z_2) = (w_1w_2, z_1z_2\langle w_1,w_2\rangle^{(p-1)/2}).$$
Let $(\tau^\star , V_{\tau^\star})$ be a Heisenberg representation of $\mathscr{H}^\star$ whose central character is the identity map $I_p$ on $\mu_p$.

For concreteness, one can realize $\tau^\star$ as follows.
Fix a polarization $$W_0 = W_2\oplus W_4$$ of $W_0$ and define a polarization $$W^\star = W^{\star +}\oplus W^{\star -}$$ by taking
$$W^{\star +} = W_1\oplus W_2 \oplus {}^gW_1$$ and
$$W^{\star -} = W_3\oplus W_4 \oplus {}^gW_3.$$
Now take $$\tau^\star = {\rm Ind}_{W^{\star +}\times\mu_p}^{\mathscr{H}^\star}(1\times I_p ).$$
Restriction of functions from $\mathscr{H}^\star$ to $W^{\star -} = W^{\star -}\times 0$ identifies the space $V_{\tau^\star}$ of $\tau^\star$ with the space $\C [W^{\star -}]$ of complex-valued functions on $W^{\star -}$.

Let $\mathscr{S}^\star = {\rm Sp}(W^\star)$ be the symplectic group of $W^\star$ and let $$\hat\tau^\star :\mathscr{S}^\star \ltimes \mathscr{H}^\star \to \GL (V_{\tau^\star})$$  be the unique extension of $\tau^\star$.  We call this the ``Weil-Heisenberg representation'' of $\mathscr{S}^\star \ltimes \mathscr{H}^\star$.  Its restriction
$$\omega^\star :\mathscr{S}^\star  \to \GL (V_{\tau^\star})$$
to $\mathscr{S}^\star$ is called the ``Weil representation of $\mathscr{S}^\star$.''



Let $\mathscr{H} = W\times \mu_p$ be the Heisenberg group associated to $W$, and embed $\mathscr{H}$ in $\mathscr{H}^\star$ in the obvious way.

In our above model for the Heisenberg representation $\tau^\star$, the space $$V_\tau = V_{\tau^\star}^{{}^gW_{1}} $$ of ${}^g W_1$-fixed vectors corresponds to the space $\C [W^- ]$ of functions in $\C [W^{\star -}]$ that are supported in the totally isotropic space $$W^-  = W_3\oplus W_4.$$
It is an irreducible summand in the decomposition of $\tau^\star |\mathscr{H}$, and restriction of functions from $\mathscr{H}^\star$ to $\mathscr{H}$ identifies $V_\tau$ with the space of the Heisenberg representation
$$\tau = {\rm Ind}_{W^+\times \mu_p}^{\mathscr{H}}(1\times I_p),$$ where $$W^+ = W_1\oplus W_2.$$

Similarly, we take 
$$V_{{}^g \tau} = V_{\tau^\star}^{W_{1}} $$ and observe that this corresponds to the space $\C [{}^gW^-]$.  Note that
$${}^g W^- = W_4\oplus {}^g W_3$$ since ${}^g W_4$ and $W_4$ are identified in $W^\star$.
Restriction of functions from $\mathscr{H}^\star$ to ${}^g\mathscr{H}$ identifies $V_{{}^g\tau}$ with the space of the Heisenberg representation
$${}^g\tau = {\rm Ind}_{{}^gW^+\times \mu_p}^{{}^g\mathscr{H}}(1\times I_p),$$ where $${}^g W^+ = {}^g W_1\oplus W_2.$$


Let $\mathscr{H}_0$ be the subgroup $$\mathscr{H}_0 =\mathscr{H}\cap {}^g\mathscr{H}= W_0\times \mu_p$$ of $\mathscr{H}^\star$.  This is the Heisenberg group associated to $W_0$.
One may view $\mathscr{H}^\star$ as a fibered sum $$\mathscr{H}^\star = \mathscr{H}\sqcup_{\mathscr{H}_0} {}^g\mathscr{H}.$$

Let $$V_{\tau_0} = V_\tau \cap V_{{}^g\tau}= V_{\tau^\star}^{W_{1}\times {}^gW_{1}}.$$
This is a common irreducible summand in the decompositions of $\tau^\star |\mathscr{H}_0$, $\tau |\mathscr{H}_0$ and ${}^g\tau|\mathscr{H}_0$ into irreducibles.
Restriction of functions from $\mathscr{H}^\star$ to $\mathscr{H}_0$ identifies $V_{\tau_0}$ with the space of 
$$\tau_0 = {\rm Ind}_{W_2\times \mu_p}^{\mathscr{H}_0}(1\times I_p).$$
Restricting functions to $W_4 = W_4\times 0$ identifies $V_{\tau_0}$ with $\C [W_4]$.

At this point, we have embedded the representations $\tau$ and ${}^g\tau$ within the representation $\tau^\star$ and shown that they are identical on the overlap of their domains $\mathscr{H}_0 =\mathscr{H}\cap {}^g\mathscr{H}$.
The explicitly exhibits that $\tau$ and ${}^g\tau$ intertwine or, in other words, that the space
${\rm Hom}_{\mathscr{H}_0}({}^g\tau ,\tau)$ is nonzero.  In fact, Yu shows that the latter space has dimension one, and hence $\tau_0$ occurs uniquely in $\tau$ and in ${}^g\tau$.

We turn now to the intertwining of the Heisenberg-Weil and Weil representations.

The symplectic group $\mathscr{S} = {\rm Sp}(W)$ embeds in $\mathscr{S}^\star$ in an obvious way.  The Weil-Heisenberg representation $\hat\tau^\star$ restricts to the Weil-Heisenberg representation $$\hat\tau :\mathscr{S}\ltimes \mathscr{H}\to \GL( V_\tau)$$ that extends $\tau$.  Let $\omega$ denote the resulting Weil representation of $\mathscr{S}$.

Similar remarks apply with the roles of $W$ and ${}^gW$ interchanged and we use the natural notations in this context.

The symplectic group $\mathscr{S}_0 = {\rm Sp}(W_0)$ has a natural embedding in $\mathscr{S}^\star$ as  a proper subgroup of $\mathscr{S}\cap {}^g\mathscr{S}$.
We obtain the Weil-Heisenberg representation $$\hat\tau_0 :\mathscr{S}_0 \ltimes \mathscr{H}_0\to \GL(V_{\tau_0})$$ by restricting any of the Weil-Heisenberg representations $\hat\tau^\star$, $\hat\tau$ or ${}^g\hat\tau$.
Let $\omega_0$ be the associated Weil representation of $\mathscr{S}_0$.

At this point, it is easy to deduce that $\omega$ and ${}^g\omega$ intertwine, however, we caution that it is not enough to simply observe 
 that $\omega$ and ${}^g\omega$ both restrict to $\omega_0$ on $V_{\tau_0}$, since $\mathscr{S}_0$ does not coincide with $\mathscr{S}\cap {}^g\mathscr{S}$.
 
 Instead, we simply observe that
the group $\mathscr{S}\cap {}^g\mathscr{S}$ stabilizes $V_{\tau_0}$ and thus
 $V_{\tau_0}$ may be viewed as a common $(\mathscr{S}\cap {}^g\mathscr{S})\ltimes \mathscr{H}_0$-submodule of $\hat\tau$ and ${}^g\hat\tau$.
 We have now shown:
 
 \begin{lemma}\label{pintertwine}
 The intertwining space ${\rm Hom}_{\mathscr{H}_0}({}^g\tau,\tau)$ is 1-dimensional and identical to
 ${\rm Hom}_{({}^g \mathscr{S}\cap \mathscr{S})\ltimes \mathscr{H}_0}({}^g \hat\tau ,\hat\tau)$.  One can associate to each complex number $c$ an intertwining operator $\mathscr{I}_c: V_{{}^g\tau}\to V_\tau$  by taking
 \begin{itemize}
\item  $\mathscr{I}_c(v) = cv$ for all  $v\in V_{\tau_0}$,
\item  $\mathscr{I}_c \equiv 0$ on all  irreducible $\mathscr{H}_0$-modules other  than $V_{\tau_0}$ occurring  $V_{{}^g\tau}$.
\end{itemize}
 \end{lemma}

\medskip
\begin{corollary}\label{Yufix}
Proposition 14.1 and Theorem 14.2 of \cite{MR1824988} are valid.
\end{corollary}

\begin{proof}
Our claim follows from Lemma \ref{pintertwine} upon pulling back our representations via a special homomorphism and observing that the image of ${}^gK\cap K$ in $\mathscr{S}^\star$ lies in ${}^g\mathscr{S}\cap \mathscr{S}$.
\end{proof}

\medskip
\begin{remark}
Note that there is a  representation of $\mathscr{S}\cap {}^g\mathscr{S}$ on $V_{\tau_0}$ defined by using the natural projection $\mathscr{S}\cap {}^g\mathscr{S} \to \mathscr{S}_0$ and then composing with $\omega_0$.
Let $\omega_0^\sharp$ denote this representation.
Let $\hat\omega_0$ be the representation of $\mathscr{S}\cap {}^g\mathscr{S}$ on $V_{\tau_0}$ given by restricting $\omega^\star$ (or $\omega$ or ${}^g\omega$).  There must be a character $\chi$ of $\mathscr{S}\cap {}^g\mathscr{S}$ such that $\hat\omega_0 = \omega_0^\sharp \otimes \chi$, but this character $\chi$ need not be trivial.  In the proof of Lemma 14.6 \cite{MR1824988}, it is incorrectly stated that ``$N_0$ is contained in the commutator subgroup of $P_0$.''  This leads to problems in the statements of Lemma 14.6 and Proposition 14.7 in \cite{MR1824988}.  These problems involve the fact that $\chi$ need not be trivial, however, as we have indicated, the nontriviality of $\chi$ is ultimately is irrelevant for the required intertwining results.
\end{remark}

\subsection{The construction of $\kappa$}

Fix a permissible representation $\rho$ of $H_x$.  Theorem \ref{ourmainresult}  states, in part,  that, 
up to isomorphism, there is a unique representation $\kappa = \kappa (\rho)$ of $K$ such that:
\begin{itemize}
\item[{\rm (1)}]  The character of $\kappa$ has support in $H_xK_+$.
\item[{\rm (2)}]  $\kappa |K_+$ is a multiple of $\hat\phi$.
\item[{\rm (3)}]  $\kappa | H_x = \rho \otimes \omega_0\otimes\cdots \otimes \omega_{d-1}$.
\end{itemize}

In this section, we construct such a representation $\kappa$.  Since Conditions (1) -- (3) completely determine the character of $\kappa$, once we have proven the necessary existence result, the uniqueness, up to isomorphism, will follow.

Since $K = H_xJ^1\cdots J^d$, and since $\kappa$ is determined on $H_x$ by Condition (3), it suffices to define $\kappa |J^{i+1}$, for all $i\in \{ 0,\dots , d-1\}$.

Now suppose $i\in \{ 0,\dots , d-1\}$.    Fix a  special homomorphism  $\nu_i :J^{i+1}\to \mathscr{H}_i$, in the sense of Definition \ref{spechom}, and use it to pull back $\tau_i$ to a representation $\tau_{\nu_i}$ of $J^{i+1}$.



 
%
%
%
%
%
%
%

Let
\begin{eqnarray*}
J^{i+1}_\flat &=& G^{i+1}_{\rm der} \cap J^{i+1}\\
J^{i+1}_{\flat ,+} &=& G^{i+1}_{\rm der} \cap J^{i+1}_+.
\end{eqnarray*}
Since $[J^{i+1},J^{i+1}]\subset J^{i+1}_\flat$, the quotient $J^{i+1}/J^{i+1}_\flat$ is a compact abelian group with subgroup $J^{i+1}_+/J^{i+1}_{\flat ,+}$.

Choose a character $\chi_{i}$ of $J^{i+1}$ that occurs as an irreducible component of the induced representation
$${\rm Ind}_{J^{i+1}_+/J^{i+1}_{\flat,+}}^{J^{i+1}/J^{i+1}_\flat} \left(  \zeta_i^{-1}(\hat\phi |J^{i+1}_+)\right).$$
Here, we are using Lemma \ref{Kplusidentity}(5).

Note that Frobenius Reciprocity implies
$$\hat\phi |J^{i+1}_+ =  (\chi_i |J^{i+1}_+)\cdot \zeta_i.$$

Define $\kappa$ on $J^{i+1}$ by
$$\kappa (j) = \chi_{i}(j) (1_\rho\otimes 1_0\otimes \cdots \otimes 1_{i-1}\otimes \tau_{\nu_i} (j)\otimes 1_{i+1}\otimes\cdots\otimes 1_{d-1}),$$
where $1_\ell$ denotes the identity map on $V_\ell$.

\begin{lemma}\label{welldefined}
The definitions of $\kappa |H_{x}$, $\kappa |J^1$, $\dots$, $\kappa |J^d$ are compatible and define a representation $\kappa : K\to \GL (V)$ such that $\kappa |K_+$ is a multiple of $\hat\phi$.
\end{lemma}

\begin{proof}
There are three things to prove:
\begin{itemize}
\item[(i)] The given definitions of $\kappa$ on  $H_x$, $J^1$, $\dots$, $J^d$ are compatible and hence yield a well-defined mapping $\kappa : K\to \GL (V)$.
\item[(ii)] The  mapping $\kappa$ is a homomorphism.
\item[(iii)] $\kappa |K_+$ is a multiple of $\hat\phi$.
\end{itemize}

It is convenient to start by showing that the definitions of $\kappa$ on  $H_x$, $J^1$, $\dots$, $J^d$ are compatible with (iii).

First, we show that the restriction of $\kappa |H_x$ to $H_x\cap K_+ =H_{x,0+}$ is a multiple of $\hat\phi |H_{x,0+} = \phi$.
By definition, $$\kappa | H_x = \rho \otimes \omega_0\otimes\cdots \otimes \omega_{d-1}$$ and $\rho | H_{x,0+}$ is a multiple of $\phi$.
So it suffices to show that each factor $\omega_i$ when restricted to $H_{x,0+}$ is a multiple of the trivial representation.  
The latter statement follows directly from the fact that $H_{x,0+}$ acts trivially by conjugation on $W_i = J^{i+1}/J^{i+1}_+$.  

(In fact, Lemma 4.2 \cite{MR1824988} implies that $[H_{x,0+},J^{i+1}]\subset J^{i+1}_{++}\subset J^{i+1}_+$ which, in turn, implies that $H_{x,0+}$ acts trivially by conjugation on $W_i = J^{i+1}/J^{i+1}_+$.  See also Remark \ref{AdlerSpice} below for more general information about commutators.)

Now fix $i\in \{ 0,\dots , d-1\}$.  The restriction of $\kappa |J^{i+1}$ to $J^{i+1}\cap K_+ =J^{i+1}_+$ is a multiple of $(\chi_i |J^{i+1}_+)\cdot \zeta_i$.   
But since $\hat\phi |J^{i+1}_+ =  (\chi_i |J^{i+1}_+)\cdot \zeta_i$, we see that the restriction of $\kappa |J^{i+1}$ to $J^{i+1}_+$ is compatible with (iii).

Next, we observe that
$$K = H_{x,0+}J^1_+\cdots J^d_+ = (H_x\cap K_+)(J^1\cap K_+)\cdots (J^d\cap K_+).$$  Therefore, if (i) and (ii) hold then (iii) must also hold, according to the previoous discussion about compatibility with (iii) of the various restrictions of $\kappa$.

It remains to prove (i) and (ii).  We prove both things using a single induction.

Suppose we are given $i\in \{ 0,\dots , d-1\}$ such that the definitions of $\kappa$ on $H_{x},J^1,\cdots , J^i$ are compatible and yield a representation $\kappa |K^i$  of $K^i = H_{x}J^1\cdots J^i$.  (When $i=0$, this amounts to the trivial assumption that $\kappa |H_{x}$ is well defined.)

To show that we have a well-defined representation $\kappa : K\to \GL (V)$, it suffices to show that the definitions of $\kappa|K^i$ and $\kappa |J^{i+1}$ are compatible, and that the resulting function $\kappa |K^{i+1}$ is a homomorphism $K^{i+1}\to \GL (V)$.

Compatibility is equivalent to the condition that $\kappa |K^i$ and $\kappa|J^{i+1}$ agree on $K^i\cap J^{i+1} = G^i_{x,r_i}$.
Suppose $i\ge 1$.  Then since $G^i_{x,r_i} = J^i\cap J^{i+1}$, compatibility reduces to the condition that $\kappa |J^i$ and $\kappa |J^{i+1}$ agree on $G^i_{x,r_i}$.
The latter condition indeed holds since both $\kappa |J^i$ and $\kappa |J^{i+1}$ restrict to a multiple of $\hat\phi |G^i_{x,r_i}$.  A  similar argument may be used when $i=0$.

It follows that we have a well-defined mapping $$\kappa |K^{i+1}: K^{i+1}\to \GL(V).$$ We now verify that this mapping is a homomorphism.

Suppose $k,k'\in K^i$ and $j,j'\in J^{i+1}$.  It is easy to see that the condition
$$\kappa (kj)\kappa (k'j') = \kappa (kjk'j')$$ is equivalent to the condition
$$\kappa (k'^{-1}jk') = \kappa (k')^{-1}\kappa (j)\kappa (k').$$
Now write $$k' = hj_1\cdots j_i,$$ with $h\in H_{x}$, $j_1\in J^1$, $\dots$, $j_i\in J^i$ and substitute for $\kappa (k')$ the product
$\kappa (h)\kappa(j_1)\cdots \kappa (j_i)$.  Then we obtain the condition
$$\chi_i (k'^{-1}jk') \, \tau_{\nu_i}(h^{-1}jh) = \chi_i (j)\, \tau_{\nu_i}(j).$$
In light of Lemma \ref{taunuiidentity}, this reduces to
$$\chi_i (k'^{-1}jk')   = \chi_i (j)$$ or, equivalently,
$$\chi_i |[K^i,J^{i+1}]=1.$$
But, by construction, $\chi_i$ is trivial on $J^{i+1}_\flat = G^{i+1}_{\rm der}\cap J^{i+1}$.
Thus our claim  follows from the fact that $[K^i,J^{i+1}]\subset J^{i+1}_\flat$.
\end{proof}

\begin{lemma}\label{kappainvariance}
The character of $\kappa$ has support in $H_xK_+$.
Consequently, up to isomorphism, the  representation $\kappa$ is independent of the choices of the $\nu_i$'s and $\chi_i$'s.  
\end{lemma}

\begin{proof}
The second assertion follows from the first one, 
since conditions (2) and (3) in Theorem \ref{ourmainresult} are independent of the choices of the $\nu_i$'s and $\chi_i$'s.

So  it suffices to prove that the character $\chi_\kappa$ of $\kappa$ has support in $H_{x}K_+$.

By definition,
$$\kappa (hj_1\cdots j_d) = \left(\prod_{i=0}^{d-1} \chi_{i}(j_{i+1})\right) \rho(h)\otimes \left( \bigotimes_{i=0}^{d-1} \omega_i(h)\, \tau_{\nu_i}(j_{i+1})\right),$$
where $h\in H_{x}$, $j_1\in J^1,\dots , j_d\in J^d$.
Taking traces of the latter finite-dimensional operators gives
$$\chi_\kappa (hj_1\cdots j_d) = \left(\prod_{i=0}^{d-1} \chi_{i}(j_{i+1})\right) \tr\rho(h)\otimes \left( \bigotimes_{i=0}^{d-1} \tr \left(\omega_i(h)\, \tau_{\nu_i}(j_{i+1})\right)\right).$$
The desired fact about the support of $\chi_\kappa$ now follows from Howe's computation of the support of the finite Heisenberg/Weil representations. (See \cite[Proposition 2]{MR0316633} and \cite[Theorem 4.4]{MR0460477}.)
\end{proof}

\bigskip
\begin{remark}
The previous proof is similar to the proof of Proposition 4.24 in \cite{MR2431732}.
\end{remark}

\bigskip
\begin{remark}
Recall that $\kappa |J^{i+1}$ is the product of a (noncanonical) character $\chi_i$ of $J^{i+1}$ with the pullback $\tau_{\nu_i}$ of a Heisenberg representation $\tau_i$ via a special homomorphism $\nu_i : J^{i+1}\to \mathscr{H}_i$.  The character $\chi_i$ is chosen in such a way that $\kappa |J^{i+1}_+$ is a multiple of $\hat\phi |J^{i+1}_+$.  
  A consequence of Lemma \ref{kappainvariance} is that the equivalence classes of the representations $\kappa |J^{i+1}$ are invariants that are canonically associated to $\rho$.   {\it A priori}, knowing the equivalence classes of $\kappa|H_x$ and the $\kappa |J^{i+1}$'s is not enough to determine the equivalence class of $\kappa$.\end{remark}

 \subsection{Irreducibility of $\pi$}
  
 We begin by recapitulating a basic fact from \cite[\S2.1]{MR2431732}.

 Suppose $\mathscr{G}$ is a totally disconnected group with center $\mathscr{Z}$, and suppose $\mathscr{K}$ is an open subgroup of $\mathscr{G}$ such that $\mathscr{K}$ contains $\mathscr{Z}$ and the quotient $\mathscr{K}/\mathscr{Z}$ is compact.
 
 Assume $(\mu,V_\mu)$ is an irreducible, smooth, complex representation of $\mathscr{K}$.  When we write ${\rm ind}_{\mathscr{K}}^{\mathscr{G}}(\mu)$, we are referring to the space of functions $$f:\mathscr{G}\to V_\mu$$ such that:
\begin{itemize}
\item $f(kg) = \mu (k)\, f(g)$, for all $k\in \mathscr{K}$, $g\in\mathscr{G}$,
\item $f$ has compact support modulo $\mathscr{Z}$ or, equivalently, the support of $f$ has finite image in $\mathscr{K}\bs \mathscr{G}$,
\item $f$ is fixed by right translations by some open subgroup $\mathscr{K}_f$ of $\mathscr{G}$.
 \end{itemize}
 We view ${\rm ind}_{\mathscr{K}}^{\mathscr{G}}(\mu)$ as a $\mathscr{G}$-module, where $\mathscr{G}$ acts on functions by right translations.
 
 A well known and fundamental fact, due to Mackey, is that the representation ${\rm ind}_{\mathscr{K}}^{\mathscr{G}}(\mu)$
 is irreducible precisely when 
$I_g(\mu)\ne 0$ implies $g\in \mathscr{K}$, where
$$I_g(\mu) = {\rm Hom}_{g \mathscr{K}g^{-1}\cap \mathscr{K}}({}^g\mu,\mu),$$
for $g\in \mathscr{G}$.
We use this  repeatedly in this section.

Now suppose $\kappa$ is a representation of $K$ as in the statement of Theorem \ref{ourmainresult} and let $\pi = {\rm ind}_K^G(\kappa)$.

\begin{lemma}\label{irred}
The representation $\pi$ is irreducible.
\end{lemma}

\begin{proof}
We may as well assume $d>0$, since there is nothing to prove if $d=0$.

Assume $g_d\in G$ and $I_{g_d}(\kappa)\ne 0$.
It suffices  to show $g_d\in K$.

We start the proof with a recursive procedure.

Assume that $i\in \{ 0,\dots , d-1\}$ and  we are given $g_{i+1}\in G^{i+1}$ such that $I_{g_{i+1}}(\kappa )\ne 0$.  We will show that there exists
$$g_i\in G^i \cap J^{i+1}g_{i+1}J^{i+1}.$$
For such $g_i$, we also show that it is necessarily the case that $I_{g_i}(\kappa)\ne 0$.

Since $\kappa |K_+$ is a multiple of $\hat\phi$, the assumption $I_{g_{i+1}}(\kappa)\ne 0$ implies that  $I_{g_{i+1}}(\hat\phi |K^{i+1}_+)\ne 0$
or, equivalently,
$$\hat\phi |([g_{i+1}^{-1},K^{i+1}_+]\cap K^{i+1}_+)=1.$$
Here, $K^{i+1}_+$ is the group
\begin{eqnarray*}
K^{i+1}_+
&=&H_{x,0+}J^1_+\cdots J^{i+1}_+\\
&=&H_{x,0+}L^1_+\cdots L^{i+1}_+\\
&=&(G^0,\dots, G^{i+1})_{x,(0+,s_0+,\dots ,s_i+)}.
\end{eqnarray*}

We have 
$$\zeta_{i} |([g_{i+1}^{-1},J^{i+1}_+]\cap J^{i+1}_+)= \hat\phi|([g_{i+1}^{-1},J^{i+1}_+]\cap J^{i+1}_+)=1,$$
since, according to Lemma \ref{Kplusidentity}(5), the characters $\zeta_i$ and $\hat\phi$ agree on $G^{i+1}_{{\rm der}}\cap J^{i+1}_+$.

According to  \cite[Theorem 9.4]{MR1824988}, there exist elements $g_i\in G^i$ and $j_{i+1},j'_{i+1}\in J^{i+1}$ such that $$g_i = j_{i+1}g_{i+1}j'_{i+1}.$$
Then $$\Lambda\mapsto \kappa (j_{i+1})\circ \Lambda\circ\kappa (j'_{i+1})$$ determines a bijection
$$I_{g_{i+1}}(\kappa)\cong I_{g_i}(\kappa).$$
Consequently, $I_{g_i}(\kappa)\ne 0$.

Thus we may continue the recursion until we have produced a sequence $g_d,\dots , g_0$, where $g_i\in G^i$ and $I_{g_i}(\kappa)\ne 0$.

In particular, we obtain $g_0\in H$ such that $$g_d \in J^d\cdots J^1 g_0 J^1\cdots J^d\subset Kg_0 K.$$  So to show $g_d$ lies in $K$, it suffices to show $g_0$ lies in $H_x = K\cap H$.

Hence, it suffices to show  that when $h\in H$ and $I_{h}(\kappa )\ne 0$  then it must be the case that $h\in H_{x}$.

So suppose $h\in H$ and $I_h (\kappa )\ne 0$.  The intertwining space $I_h (\kappa ) $
 consists of linear endomorphisms $\Lambda$ of the space $V$ of $\kappa$ such that
$$\Lambda ({}^h \kappa (k)v) =\kappa (k)\, \Lambda(v),\quad \forall k\in hK h^{-1}\cap K .$$

Given such $\Lambda\in I_h (\kappa )$, there is an associated
$$\lambda\in {\rm Hom}_{hKh^{-1}\cap K }({}^h \kappa\otimes \tilde\kappa ,1),$$ where $hKh^{-1}\cap K$ is embedded diagonally in $K \times K$ and $(\tilde\kappa, \widetilde V)$ is the contragredient of $(\kappa,V)$.  It is defined on elementary tensors in $V \otimes \widetilde V$ by
$$\lambda (v\otimes \tilde v) = \langle \Lambda (v),\tilde v\rangle.$$
The map $\Lambda\mapsto \lambda$ gives a linear isomorphism
$$I_h (\kappa) \cong {\rm Hom}_{hKh^{-1}\cap K }({}^h\kappa\otimes \tilde\kappa ,1).$$
The latter remarks apply very generally to intertwining, not just to $\kappa$, and we will use them for various representations.

We observe that for each $i\in \{ 0,\dots , d-1\}$ the intertwining space  $I_h (\tau_{\nu_i})$ has dimension one, according to   \cite[Proposition 12.3]{MR1824988}.  For each such $i$, fix a nonzero element $\Lambda_i\in I_h (\tau_{\nu_i})$ and let 
$\lambda_i$ be the corresponding element of ${\rm Hom}_{hJ^{i+1}h^{-1}\cap J^{i+1} }({}^h \tau_{\nu_i}\otimes \tilde\tau_{\nu_i} ,1)$.

Suppose
$$v= v_\rho\otimes v_0\otimes \cdots \otimes v_{d-1}\in V_\rho\otimes V_0\otimes \cdots \otimes V_{d-1} = V$$
and
$$\tilde v=   \tilde v_\rho \otimes \tilde v_0\otimes \cdots \otimes \tilde v_{d-1}\in \widetilde V_\rho\otimes \widetilde V_0\otimes \cdots \otimes \widetilde V_{d-1} = \widetilde V.$$

Suppose $i\in \{ 0,\dots ,d-1\}$.
Fix all of the components of $v$ and $\tilde v$ except for $v_i$ and $\tilde v_i$ and define a linear form $\lambda'_i :V_i\otimes \widetilde V_i\to \C$ by
$$\lambda'_i (v_i\otimes \tilde v_i) = \lambda (v\otimes \tilde v).$$
Then for $j\in hJ^{i+1}h^{-1}\cap J^{i+1}$ we have
\begin{eqnarray*}
\lambda'_i (v_i\otimes \tilde v_i)
&=&
\lambda  (\kappa (h^{-1}jh)v\otimes \tilde\kappa (j)\tilde v)\\
&=&
\chi_i (h^{-1}jhj^{-1}) \lambda'_i (\tau_{\nu_i} (h^{-1}jh)v_i\otimes \tilde\tau_{\nu_i} (j)\tilde v_i)\\
&=& \lambda'_i (\tau_{\nu_i} (h^{-1}jh)v_i\otimes \tilde\tau_{\nu_i} (j)\tilde v_i),
\end{eqnarray*}
where the triviality of $\chi_i (h^{-1}jhj^{-1})$ follows from the fact that $\chi_i$ is trivial on $J^{i+1}_\flat = G^{i+1}_{\rm der}\cap J^{i+1}$.

Therefore, $\lambda'_i$ lies in $${\rm Hom}_{hJ^{i+1}h^{-1}\cap J^{i+1} }({}^h \tau_{\nu_i}\otimes \tilde\tau_{\nu_i} ,1) = \C \lambda_i.$$

Using a ``multiplicity one'' argument  as in the proof of Lemma 5.24 \cite{MR2431732}, one can show that 
for every linear form $$\lambda\in {\rm Hom}_{hKh^{-1}\cap K }({}^h\kappa\otimes \tilde\kappa ,1)$$ there exists a unique linear form
$$\lambda_\rho\in {\rm Hom}(V_\rho\otimes \widetilde V_\rho ,\C)$$ such that
$$\lambda (v\otimes \tilde v) = \lambda_\rho (v_\rho \otimes \tilde v_{\rho})\cdot\prod_{i=0}^{d-1} \lambda_i (v_i\otimes \tilde v_i).$$

Roughly speaking, the argument goes as follows.  Fix $v_{d-1}$ and $\tilde v_{d-1}$ so that $\lambda_{d-1}(v_{d-1}\otimes \tilde v_{d-1})$ is nonzero.
Define 
$$\partial v = v_\rho\otimes v_0\otimes \cdots \otimes v_{d-2}$$ in $$\partial V = V_\rho\otimes V_0\otimes \cdots \otimes V_{d-2}$$
and
$$\partial\tilde v =\tilde v_\rho \otimes \tilde v_0\otimes \cdots \otimes \tilde v_{d-2}$$
in $$ \partial\widetilde V = \widetilde V_\rho\otimes \widetilde V_0\otimes \cdots \otimes \widetilde V_{d-2}.$$ 
Then there is a linear form on $\partial V\otimes \partial \widetilde V$ given on elementary tensors by 
$$\partial\lambda 
(\partial v\otimes\partial\tilde v) = \frac{\lambda (v\otimes \tilde v)}{\lambda_{d-1}(v_{d-1}\otimes \tilde v_{d-1})}.
$$

Repeating this procedure one coordinate at a time, one gets a sequence $\partial^j\lambda  $ of linear forms on 
$\partial^j V\otimes \partial^j \widetilde V$ 
defined on elementary tensors by 
$$\partial^j\lambda 
(\partial^j v\otimes\partial^j\tilde v) 
= \frac{\lambda (v\otimes \tilde v)}{\prod_{i=1}^{j}\lambda_{d-i}(v_{d-i}\otimes \tilde v_{d-i})}.
$$
This ultimately yields the desired linear form $\lambda_\rho$.

We claim that
$$\lambda_\rho \in {\rm Hom}_{hH_{x}h^{-1}\cap H_{x}}({}^h\rho\otimes\tilde\rho ,1).$$  Indeed, if $k\in hH_{x}h^{-1}\cap H_{x}$ then, according to  Corollary \ref{Yufix}, we have
\begin{eqnarray*}
\lambda  ( v \otimes  \tilde v)&=&
\lambda  (\kappa  (h^{-1}kh)v \otimes \tilde\kappa (k)\tilde v)\\
&=&\lambda_\rho (\rho(h^{-1}kh)v_\rho \otimes \tilde\rho (k)\tilde v_\rho)\prod_{i=0}^{d-1}
\lambda_i (\omega_i (h^{-1}kh)v_i \otimes \tilde\omega_i (k)\tilde v_i)\\
&=&\lambda_\rho (\rho(h^{-1}kh)v_\rho \otimes \tilde\rho (k)\tilde v_\rho)\prod_{i=0}^{d-1}
\lambda_i ( v_i \otimes \tilde v_i).
\end{eqnarray*}


Since $\lambda_\rho$ is necessarily nonzero, there must be a corresponding nonzero element $\Lambda_\rho \in  I_h (\rho)$.    Since $I_h(\rho)$ is nonzero and  $\rho$ induces an irreducible representation of $H$, we deduce that $h$ lies in $H_{x}$, which completes the proof.
\end{proof}

\bigskip
\begin{remark}
In Proposition 4.6 \cite{MR1824988}, it is shown that  cuspidal $G$-data that satisfy conditions ${\bf SC1}_i$, ${\bf SC2}_i$, and ${\bf SC3}_i$ of  \cite{MR1824988} yield irreducible (hence supercuspidal) representations of $G$.  In Theorem 15.1 in \cite{MR1824988}, it is shown that generic, cuspidal $G$-data satisfy the {\bf SC} conditions.  Theorem 9.4, Theorem 11.5, and Theorem 14.2 of \cite{MR1824988}, respectively, are the results that specifically show that generic data satisfy the conditions ${\bf SC1}_i$, ${\bf SC2}_i$, and ${\bf SC3}_i$, respectively.  Our proof of Lemma \ref{irred} uses  Theorem 9.4  \cite{MR1824988} and Theorem 14.2 of \cite{MR1824988} (with the revised proof in Corollary \ref{Yufix} above).
\end{remark}

\bigskip
\begin{remark}
As just indicated, our proof of Lemma \ref{irred} is based partly on Yu's proof of  his Proposition 4.6.   But  in the latter part of his proof, Yu uses an inductive argument due to Bushnell-Kutzko \cite[\S5.3.2]{MR1204652}.  
By contrast, we have adapted the proof of Lemma 5.24 \cite{MR2431732}.  
(The latter result was not intended to handle intertwining issues, but it is interesting  that it can be used for this purpose.)
\end{remark}

\section{The connection with Yu's construction}\label{sec:YuisoftenMe}

The theory of generic cuspidal $G$-data was introduced by Yu in \cite{MR1824988}, but in discussing this theory we follow the notations and terminology of \cite[pp.~50, 52, 56]{MR2431732}.

Fix a generic cuspidal $G$-datum $\Psi = (\vec\bG , y,\rho_0 ,\vec\phi )$ and a torus $\bT$ as in \cite{MR2431732}.
Then Yu's construction associates to $\Psi$ a representation $\kappa (\Psi)$ of an open compact-mod-center subgroup $K(\Psi)$ of $G$.   The representation $\kappa (\Psi)$ induces an irreducible, supercuspidal representation $\pi (\Psi)$ of $G$.   (See \cite[page 68]{MR2431732}.)

Let $\bG = \bG^d$, $\bH = \bG^0$, $\bT$ and $y$ come from the given datum $\Psi$, and let 
$$\rho = \rho(\Psi)= \rho_0 \otimes \prod_{i=0}^d (\phi_i |H_{x} ),$$ where $x$ is the vertex in $\mathscr{B}_{\rm red} (\bH,F)$ associated to $y$.

Let $\vec r = (r_0,\dots , r_d)$ be the sequence of depths associated to $\Psi$ (as in \cite{MR2431732}).  Let $$\phi = \prod_{i=0}^d (\phi_i|H_{x,0+}).$$
Then  $\phi$ and $\phi_i$  agree on $[G^{i+1},G^{i+1}]\cap H_{x,r_{i-1}+}$.

For each $i\in \{ 0,\dots , d-1\}$, we have a symplectic space $W_i$ associated to $\Psi$ as in \cite[pp.~66]{MR2431732}.
Associated to $W_i$, we can define a Heisenberg group structure on 
$\mathscr{H}_i = W_i\times \mu_p$, as in \S\ref{sec:HeisWeil}.

Define another Heisenberg group $W_i^\sharp$ by letting $$W_i^\sharp = W_i \times (J^{i+1}_+/\ker\zeta_i)$$ with multiplication
defined as in \cite[pp.~22, 59, 67]{MR2431732}.
Then $(w,z)\mapsto (w,\zeta_i(z))$ defines an isomorphism $W^\sharp_i\cong \mathscr{H}_i$ of Heisenberg groups.

Fix a Heisenberg representation $(\tau_i ,V_i)$ of $\mathscr{H}_i$.  Let $(\tau_i^\sharp , V_i)$ be the Heisenberg representation of $W_i^\sharp$ obtained via the given isomorphism $W^\sharp_i\cong \mathscr{H}_i$.

Extending $\tau_i$ as in \S\ref{sec:HeisWeil}, we obtain a Weil representation $\mathscr{S}_i\to \GL(V_i)$.  This is the same as the Weil representation that comes from $\tau_i^\sharp$ in \cite[p.~67]{MR2431732}.

We have a map $H_{x}\to \mathscr{S}_i$ given by conjugation.  Pulling back the Weil representation via the latter map yields a representation $$\omega_i : H_{x}\to \GL(V_i).$$

In Yu's construction and ours, we must arbitrarily choose  for each $i\in \{ 0,\dots ,d-1\}$ a Heisenberg representation within a prescribed isomorphism class.  
Implicit in the statement of the next result is the assumption that we choose the same family of Heisenberg representations when we construct $\kappa (\Psi)$ and $\kappa (\rho)$.
More precisely, we first choose a family of Heisenberg representations $\tau_i^\sharp$ as in \cite[p.~67]{MR2431732}, then, after it is established that $\vec\bG (\Psi) = \vec\bG(\rho)$ and $\vec r(\Psi) = \vec r(\rho)$, we use in the construction of $\kappa (\rho)$ the Heisenberg representations $\tau_i$ associated to the $\tau_i^\sharp$'s as above.

\bigskip
\begin{lemma}\label{IamYu}
Suppose $\Psi$ is a generic cuspidal $G$-datum and  $\rho$ is the corresponding representation of $H_{x}$.  Then:
\begin{itemize}
\item $\rho$ is a permissible representation.
\item The given sequences $\vec\bG = (\bG^0,\dots, \bG^d)$ and $\vec r = (r_0,\dots, r_d)$ associated to $\Psi$ are identical to the corresponding sequences associated to $\rho$.
\item The representations $\kappa (\rho)$ and $\kappa (\Psi)$ are  defined on the same group $K$ acting on the same space $V$ and the two representations are identical on the subgroups $H_{x}$, $J^1_+$, $\dots$, $J^d_+$.
\item The characters of both representations have support in $H_{x}J^1_+\cdots J^d_+$.
\end{itemize}
Consequently, the characters of $\kappa (\Psi)$ and $ \kappa (\rho)$ are identical and thus these representations of $K$ (and the associated tame supercuspidal representations of $G$) are equivalent.
\end{lemma}

\begin{proof}
Fix $\Psi = (\vec\bG , y,\rho_0 ,\vec\phi )$ and  $\bT$, as above.
We may as well replace $y$ in the datum with the point $x=[y]$, since only $x$ is relevant to Yu's construction.  As noted in \cite[\S3.5]{KalYu}, if $\bG^\sharp$ is $z$-extension of $\bG$ as in \S\ref{sec:zext}, then $\Psi$ pulls back to a generic, cuspidal $G^\sharp$-datum $\Psi^\sharp = (\vec\bG^\sharp, x,\rho_0^\sharp, \vec\phi^\sharp)$ such that Yu's representation $\pi (\Psi)$ of $G$ pulls back to his representation $\pi (\Psi^\sharp)$ of $G^\sharp$.
Therefore, we may as well assume that $\bG_{\rm der}$ is simply connected.

(Note that the complications regarding the definition of the $r_i$'s in \S\ref{sec:recovery} essentially disappear when considering Yu's construction, since Yu imposes more restrictive conditions on his objects.  For example, a given factor $\phi_i$ always has the same depth as its pullback $\phi_i^\sharp$, since we have surjections $G^{i,\sharp}_{x,r}\to G^i_{x,r}$ for all $r\ge 0$, according to Lemma 3.5.3 \cite{KalYu}.)

Let
$$\rho = \rho(\Psi)= \rho_0 \otimes \prod_{i=0}^d (\phi_i |H_{x} ).$$
In our construction, we assume that there is a character $\phi$ of $H_{x,0+}$ such that $\rho|H_{x,0+}$ is a multiple of $\phi$.  Such a character $\phi$ exists in the present situation and, according to \cite[Corollary 3.28]{MR2431732}, it is given by $$\phi = \prod_{j=0}^d(\phi_j|H_{x,0+}).$$

In general, there are various objects, such as the $\bG^i$'s, associated to $\Psi$ and, in our construction, there are similar objects associated to $\phi$.  
We need to show that the objects $\vec\bG$, $\vec r$ and $\vec\phi$ coming from $\Psi$ can be recovered from
$\phi$ exactly as in our construction.

Assume $i\in \{ 1,\dots , d-1\}$.
Let $\bZ^i$ denote the center of $\bG^i$ and let $\bZ^{i,i+1} = (\bG^{i+1}_{\rm der}\cap \bZ^i)^\circ$.
Choose a $G^{i+1}$-generic element $Z^*_i\in \mathfrak{z}^{i,*}_{-r_i}$ of depth $-r_i$ that represents $\phi_i |G^i_{x,r_i}$.
Using the decomposition
$$\mathfrak{z}^{i,*}_{-r_i} = \mathfrak{z}^{i+1,*}_{-r_i}\oplus \fr{z}^{i,i+1,*}_{-r_i},$$
we write $Z^*_i = Y^*_i +X^*_i$.
Then $Y^*_i$ represents the trivial character of $H^i_{x,r_i}$ since $\mathfrak{z}^{i+1,*}$ is orthogonal to $\mathfrak{g}^{i+1}_{\rm der}$.  Therefore, $Z^*_i$ and $X^*_i$ represent the same character of $H^i_{x,r_i}$, namely, $\phi_i |H^i_{x,r_i}$.

Since we assume $\bG_{\rm der}$ is simply connected, it follows from Lemma \ref{goodp} (with $\bH$ replaced by $\bG^{i+1}$) that 
$$[G^{i+1},G^{i+1}] \cap H_{x,r_i} = H^i_{x,r_i}$$
and, consequently, $\phi  |H^i_{x,r_i}=\phi_i |H^i_{x,r_i}$.  Therefore, $X^*_i$ must lie in the dual coset $(\phi |Z^{i,i+1})^*\in \fr{z}^{i,i+1,*}_{-r_i:(-r_i)+}$.

%

We observe now that, just as in our construction,
 $r_i$ must be the depth of $\phi |H^i_{x,0+}$, and $\bG^i$ is the unique maximal subgroup of $\bG$ containing $\bH$ such that $\bG^i$ is defined over $F$ and is an $E$-Levi subgroup of $\bG$, and $\phi |H^{i-1}_{x,r_i}=1$.

Thus, we have established  that the sequences $\vec\bG$ and $\vec r$ associated to $\Psi$ and $\phi$ coincide.  It is also follows form the remarks in \S\ref{sec:HisG} that $\rho$ is permissible.
The fact that $\vec\bG$ and $\vec r$ coincide for $\Psi$ and $\rho$ implies that the associated subgroups of $G$, such as $K$, are  the same for $\Psi$ and $\phi$.

It is routine to verify from the definitions that $\kappa (\Psi)|H_{x}$ and $\kappa (\rho)|H_{x}$ are the identical representation of the group $H_{x}$ on the same space $V$.  
We refer the reader to \cite[pp.~66--68]{MR2431732} for the definition of $\kappa (\Psi)$ and we note that it is easy to see that the representation $\omega_i$ defined just before the statement of the present lemma is given by $\omega_i (h) = \hat\tau_i^\sharp (f'_i(h))$, in the notation of in \cite[p.~67]{MR2431732}.

We also observe that both $\kappa (\Psi)$ and $\kappa (\rho)$ act according to a multiple of the character $\hat\phi$ on each of the subgroups $J^{i+1}_+$.
In the case of $\kappa (\Psi)$, this is \cite[Corollary 3.28]{MR2431732}.

The fact that the support of the character of $\kappa (\Psi)$ is contained in $H_{x}J^1_+\cdots J^d_+$ is shown in the proof of \cite[Proposition 4.24]{MR2431732}.  The corresponding fact for $\kappa (\rho)$ is our Lemma \ref{kappainvariance}.
\end{proof}

\section{Distinguished cuspidal repre\-sen\-ta\-tions for $p$-adic groups and  finite groups of Lie type}\label{sec:finitestuff}

The motivation for this paper was a desire to unify the $p$-adic and finite field theories of distinguished cuspidal representations.  We now state in a uniform way a theorem that applies both to  $p$-adic and finite fields.   The proof appears in a companion paper \cite{ANewHM}.

In this section, we simultaneously address two cases that we refer to as ``the $p$-adic case'' and ``the finite field case.''
In the $p$-adic case, $F$ is, as usual, a finite extension of $\Q_p$ with $p$ odd.  In the finite field case, $F= \mathbb{F}_{q}$ where $q$ is a power of an odd prime $p$.  Let $\bG$ be a connected, reductive $F$-group and let $G = \bG (F)$.

In the $p$-adic case, we let $\rho$ be a permissible representation of $H_{x}$.  In the finite field case, we let $\rho$ be a character in general position of $\bT (\mathbb{F}_q)$, where $\bT$ is an $F$-elliptic maximal $F$-torus.
For the sake of unity, we let $L$ denote $H_{x}$ in the $p$-adic case and $\bT (\mathbb{F}_q)$ in the finite field case.

Let $\pi (\rho)$ be the irreducible supercuspidal or cuspidal Deligne-Lusztig representation of $G$ associated to $\rho$.  Let $\mathscr{I}$ be the set of $F$-automorphisms of $\bG$ of order two, and let $G$ act on $\mathscr{I}$ by
$$g\cdot \theta = {\rm Int}(g)\circ \theta \circ {\rm Int}(g)^{-1},$$ where ${\rm Int}(g)$ is conjugation by $g$.  Fix a $G$-orbit $\Theta$ in $\mathscr{I}$.
Given $\theta\in \mathscr{O}$, let $G_\theta$ be the stabilizer of $\theta$ in $G$.  Let $G^\theta$ be the group of fixed points of $\theta$ in $G$.  
When $\vartheta$ is an $L$-orbit in $\Theta$, let $${\rm m}_L (\vartheta) = [G_\theta :G^\theta (G_\theta\cap L)],$$ for some, hence all, $\theta\in \vartheta$.

Let $\langle \Theta,\rho\rangle_G$ denote the dimension of the space ${\rm Hom}_{G^\theta}(\pi (\rho),1)$  of $\C$-linear forms on the space of $\pi (\rho)$ that are invariant under the action of $G^\theta$ for some, hence all, $\theta\in \Theta$.

For each $\theta$ such that $\theta (L) = L$, we define a character $$\varepsilon_{{}_{L,\theta}} :L^\theta \to \{ \pm 1\}$$ as follows.

In the finite field case,
$$\varepsilon_{{}_{L,\theta}} (h) = \det \left({\rm Ad}(h)|\mathfrak{g}^\theta \right).$$
One can show that this is the same as the character $\varepsilon$
defined in \cite[page 60]{MR1106911}.

In the $p$-adic case,
$$\varepsilon_{{}_{L,\theta}} (h) = \prod_{i=0}^{d-1} \left(\frac{\det\nolimits_{{\frak f}} ({\rm Ad}(h)\, |\, {\frak W}_i^+)}{\fr{P}_F}\right)_2,$$
where our notations are as follows.

First, we take
$$\mathfrak{W}_i^+
=\bigg(\bigg(\bigoplus_{a\in \Phi^{i+1}-\Phi^i}\bfr{g}_a\bigg)^{{\rm Gal}(\overline{F}/F)}\bigg)^\theta_{x,s_i:s_i+},$$ viewed as a vector space over the residue field $\fr{f}$ of $F$.
In other words, $\mathfrak{W}_i^+$ is, roughly speaking,  the space of $\theta$-fixed points in the Lie algebra of $W_i= J^{i+1}/J^{i+1}_+$. 

Next, for $u\in \fr{f}^\times$, we let $(u/\fr{P}_F)_2$ denote the quadratic residue symbol.  This is related to the ordinary Legendre symbol   by
$$\left(\frac{u}{\fr{P}_F}\right)_2
=\left(\frac{N_{\fr{f}/\mathbb{F}_p} (u)}{p}\right) = (N_{\fr{f}/\mathbb{F}_p} (u))^{(p-1)/2}= u^{(q_F-1)/2}.$$

We remark that in the $p$-adic case, $\varepsilon_{{}_{L,\theta}}$ is the same as the character $\eta'_\theta$ defined in \cite[\S5.6]{MR2431732}.

In both the finite field and $p$-adic cases, one can, at least in most cases, compute the determinants arising in the definition of $\varepsilon_{{}_{L,\theta}}$ and provide  elementary expressions for them in terms of ${\rm Gal}(\overline{F}/F)$-orbits of roots.

When $\vartheta$ is an $L$-orbit in $\Theta$, we write $\vartheta \sim \rho$ when $\theta (L) = L$ for some, hence all $\theta\in \vartheta$, and when 
the space ${\rm Hom}_{L^\theta} (\rho ,\varepsilon_{{}_{L,\theta}})$ is nonzero.
When $\vartheta\sim \rho$, we define
$$\langle \vartheta ,\rho\rangle_L = \dim {\rm Hom}_{L^\theta} (\rho ,\varepsilon_{{}_{L,\theta}}),$$ where $\theta$ is any element of $\vartheta$.  (The choice of $\theta$ does not matter.)

With these notations, we may have the following:

\medskip
\begin{theorem}\label{motivation} $\langle \Theta,\rho\rangle_G = \sum_{\vartheta\sim \rho} {\rm m}_L(\vartheta)\, \langle \vartheta , \rho\rangle_L$.\end{theorem}

\bigskip
Note that in  the special case in which
 \begin{itemize}
\item $\bG$ is a product $\bG_1\times\bG_1$,
\item $\cO$ contains the involution $\theta (x,y) = (y,x)$, 
\item $\rho$ has the form $\rho_1\times \tilde\rho_2$,
\end{itemize}
where $\tilde\rho_2$ is the contragredient of $\rho_2$,
we have
$$\langle \Theta,\rho \rangle_G = \dim {\rm Hom}_{G}(\pi (\rho_1),\pi(\rho_2)).$$
In the finite field case, our theorem in the latter situation is consistent with  the  Deligne-Lusztig inner product formula \cite[Theorem 6.8]{MR0393266}.  (See \cite[page 58]{MR1106911}, for more details.)

\bibliographystyle{amsalpha}
\bibliography{JLHrefs}

\end{document}